\documentclass[11pt]{amsart}

\usepackage{amsmath, amssymb, xypic}

\newcommand{\bbN}{{\mathbb N}}
\newcommand{\bbT}{{\mathbb T}}
\newcommand{\bbR}{{\mathbb R}}
\newcommand{\bbC}{\mathbb C}
\newcommand{\cJ}{{\mathcal J}}

\newcommand{\cK}{{\mathcal K}}
\newcommand{\p}{{\mathcal P}}
\newcommand{\cX}{{\mathcal X}}
\newcommand{\calH}{{\mathcal H}}
\newcommand{\cY}{{\mathcal Y}}
\newcommand{\cZ}{{\mathcal Z}}
\newcommand{\bA}{{\mathbf A}}

\newcommand{\cA}{{\mathcal A}}
\newcommand{\cE}{{\mathcal E}}

\newcommand{\rs}{\restriction}

\DeclareMathOperator{\Ad}{Ad}

\renewcommand{\>}{\rangle}

\newcommand{\cC}{\mathcal C}
\newcommand{\cF}{\mathcal F}

\newcommand{\cB}{\mathcal B}
\newcommand{\calD}{\mathcal D}

\newcommand{\e}{\varepsilon}
\newtheorem{thrm}{Theorem}
\newtheorem{thm}{Theorem}[section]
\newtheorem{theorem}[thm]{Theorem}

\newtheorem{claim}[thm]{Claim}

\newtheorem{lemma}[thm]{Lemma}
\newtheorem{prop}[thm]{Proposition}

\theoremstyle{definition}

\newtheorem{problem}[thm]{Problem}

\DeclareMathOperator{\Fin}{Fin}

\newcounter{my_enumerate_counter}
\newcommand{\pushcounter}{\setcounter{my_enumerate_counter}{\value{enumi}}}
\newcommand{\popcounter}{\setcounter{enumi}{\value{my_enumerate_counter}}}

\def\rs{\restriction}

\DeclareMathOperator{\SA}{sa}
 
\DeclareMathOperator{\Tr}{Tr}
 
\DeclareMathOperator{\proj}{proj} 
 \DeclareMathOperator{\SPAN}{\overline{span}}
\DeclareMathOperator{\dom}{dom}

\newcommand{\lbl}{\label}

\newcommand{\bfD}{{\mathbf D}}

\newcommand{\bfH}{{\mathbf H}}

\newcommand{\bbfR}{{\mathbf {\bar R}}}
\newcommand{\bfR}{{\mathbf R}}

\newcommand{\cP}{{\mathcal P}}

\newtheorem*{TAdef}{TA}

\newcommand{\cU}{{\mathcal U}}
\newcommand{\cV}{\mathcal V}

\newcommand{\ba}{\mathbf a}
\newcommand{\bb}{\mathbf b}
\newcommand{\bc}{\mathbf c}
\newcommand{\bP}{\mathbf P}
\newcommand{\bfQ}{{\mathbf Q}}
\newcommand{\bfS}{{\mathbf S}}
\newcommand{\bQ}{{\mathbf Q}}

\DeclareMathOperator{\even}{\text{even}}
\DeclareMathOperator{\odd}{\text{odd}}

\newcommand{\NupN}{\bbN^{\uparrow\bbN}}
\newcommand{\cUN}{\bbT^{\bbN}}
\newcommand{\cI}{\mathcal I}
\newcommand{\bK}{\mathbf K}

\newcommand{\fd}{\mathfrak d}
\newcommand{\bU}{\mathbf U}

 \DeclareMathOperator{\findex}{index}
\title[All automorphisms of $\cC(H)$ are inner]{All automorphisms of
the Calkin algebra\\  are inner}

\dedicatory{Dedicated to my wife Tatiana Velasevic and Dr. Carl J. Vaughan and Dr.
Leonard N. Girardi of New York-Presbyterian Hospital. Without them I would not
be around to prove Theorem~\ref{T.0}.}

\author{Ilijas Farah}

\address{Department of Mathematics and Statistics\\
York University\\
4700 Keele Street\\
North York, Ontario\\ Canada, M3J 1P3\\
and Matematicki Institut, Kneza Mihaila 34, Belgrade, Serbia}

\email{ifarah@mathstat.yorku.ca}
 \subjclass{46L40, 
46L05, 
 03E75, 
03E65 
}

\urladdr{http://www.math.yorku.ca/$\sim$ifarah}

\date{\today.}
\begin{document}
\begin{abstract} We prove that it is relatively consistent with the usual
axioms of mathematics that all automorphisms of the Calkin algebra are inner.
Together with a 2006 Phillips--Weaver construction of an outer automorphism
using the Continuum Hypothesis, this gives a complete solution to a 1977
problem of Brown--Douglas--Fillmore. We also give a simpler and self-contained
proof of the Phillips--Weaver result.
\end{abstract}

\maketitle

 Fix a separable infinite-dimensional complex
Hilbert space $H$. Let $\cB(H)$ be its algebra of bounded linear operators,
$\cK(H)$ its ideal of compact operators and $\cC(H)=\cB(H)/\cK(H)$ the Calkin
algebra. Let $\pi\colon \cB(H)\to \cC(H)$ be the quotient map.  In
\cite[1.6(ii)]{BrDoFi:Extensions} (also \cite{Sak:Pure}, \cite{We:Set}) it was
asked whether all automorphisms of the Calkin algebra are inner. Phillips and
Weaver (\cite{PhWe:Calkin}) gave a partial answer by constructing an outer
automorphism using the Continuum Hypothesis. We complement their answer by
showing that a well-known set-theoretic axiom  implies all automorphisms are
inner. Neither the statement of this axiom nor the proof of Theorem~\ref{T.0}
involve set-theoretic considerations beyond the standard functional analyst's
toolbox.

\begin{thrm}\lbl{T.0} Todorcevic's Axiom, TA,   implies that all
automorphisms of the Calkin algebra  of a separable Hilbert space are inner.
\end{thrm}

Todorcevic's Axiom (also known as the Open Coloring Axiom, OCA) is
stated in \S\ref{S.OCA}. Every model of ZFC has a forcing extension in 
which TA holds (\cite{Ve:Applications}). 
 TA also holds in Woodin's canonical model for negation of the Continuum
Hypothesis (\cite{Wo:CH2}, \cite{La:Showing}) and it follows from the
Proper Forcing Axiom, PFA (\cite{To:Partition}). The latter is a
strengthening of the Baire Category Theorem and besides its
applications to the theory of liftings it can be used to find other
combinatorial reductions (\cite[\S8]{To:Partition}, \cite{Moo:Five}).

 The Calkin
algebra provides both a natural context and a powerful tool for studying
compact perturbations of operators on a Hilbert space. The original motivation
for the problem solved in Theorem~\ref{T.0} comes from a classification problem
for normal operators.
 By results of Weyl, von Neumann, Berg and Sikonia, if
$a$ and $b$ are normal operators in $\cB(H)$ then one is untarily equivalent to
a compact perturbation of the other if and only if their essential spectra
coincide (see  the introduction to \cite{BrDoFi:Unitary} or \cite[\S
IX]{Dav:C*}). The \emph{essential spectrum}, $\sigma_e(a)$, of $a$ is the set
of all accumulation points of its spectrum $\sigma(a)$, together with all of
its isolated points of infinite multiplicity.  It is known to be equal to the
spectrum of $\pi(a)$ in the Calkin algebra. Therefore the map $a\mapsto
\sigma_e(a)$ provides a complete invariant for the unitary equivalence of those
operators in the Calkin algebra that lift to normal operators in $\cB(H)$.

An operator $a$ is said to be \emph{essentially normal} if $aa^*-a^*a$ is
compact, or equivalently, if its image in the Calkin algebra  is normal.
 Not every essentially normal operator is a compact perturbation of a  normal
operator. For example, an argument using Fredholm index shows that the
unilateral shift $S$ is not a compact perturbation of a normal operator
(\cite{BrDoFi:Unitary}) while its image in $\cC(H)$ is clearly a unitary. Since
the essential spectra of $S$ and its adjoint are both equal to the unit circle,
the above mentioned classification does not extend to all normal operators in
$\cC(H)$.
 For an essentially normal operator $a$ and
$\lambda\in \bbC\setminus \sigma_e(a)$ the operator $a-\lambda I$ is Fredholm.
In  \cite{BrDoFi:Unitary} (see also \cite{BrDoFi:Extensions} or \cite[\S
IX]{Dav:C*}) it was proved that the  function $\lambda\mapsto \findex
(a-\lambda I)$ together with $\sigma_e(a)$ provides a complete invariant for
the relation of unitary equivalence modulo a compact perturbation on
essentially normal operators.

It is interesting to note that the unitary equivalence of normal (even
self-adjoint) operators is of much higher complexity than the unitary
equivalence of normal (or even essentially normal) operators modulo the compact
perturbation. By the above, the latter relation is \emph{smooth}: a complete
invariant is given by a Borel-measurable map into a Polish space. On the other
hand, the complete invariant for the former given by the spectral theorem
 is of much higher complexity. As a matter of fact, in
\cite{KecSof:Strong} it was proved that the unitary equivalence of self-adjoint
operators does not admit any effectively assigned complete invariants coded by
countable structures.

Instead of the unitary equivalence modulo compact perturbation, one may
consider a coarser relation  which we temporarily denote by $\sim$. Let $a\sim
b$ if there is an automorphism $\Phi$ of the Calkin algebra sending $\pi(a)$ to
$\pi(b)$.  It is clear that $a\sim b$ implies $\sigma_e(a)=\sigma_e(b)$, and
therefore two relations coincide on normal operators. By \cite{BrDoFi:Unitary}
these two relations coincide on normal operators, and
 the conclusion of Theorem~\ref{T.0} implies that they
coincide on all of $\cB(H)$. The outer automorphism $\Phi$ constructed in
\cite{PhWe:Calkin}, as well as the one in \S\ref{S.CH} below, is
\emph{pointwise inner}:  $\Phi(\pi(a))=\Phi(\pi(b))$ implies an inner
automorphism sends $\pi(a)$ to $\pi(b)$. It is not known whether $\sim$ can
differ from the unitary equivalence modulo a compact perturbation
in some model
of set theory. In particular, it is still open whether the Continuum Hypothesis
implies the existence of an automorphism of the Calkin algebra sending the
image of the unilateral shift to its adjoint.
 See \cite{PhWe:Calkin} for a discussion and related open problems.

Theorem~\ref{T.0} belongs to a line of results starting with Shelah's
ground-breaking construction of a model of set theory in which all
automorphisms of the quotient Boolean algebra $\cP(\bbN)/\Fin$ are
trivial (\cite{Sh:Proper}). An equivalent reformulation states that
it is impossible to construct a nontrivial automorphism of
$\cP(\bbN)/\Fin$ without using some additional set-theoretic axiom.
Through the work of Shelah--Stepr\=ans, Velickovic, Just, and the
author this conclusion was extended to many other quotient algebras
$\cP(\bbN)/\cI$. The progress was made possible by replacing Shelah's
intricate forcing construction by the PFA (\cite{ShSte:PFA}) and then
in \cite{Ve:OCA} by Todorcevic's Axiom (\cite[\S8]{To:Partition}) in
conjunction with the Martin's Axiom. A survey of these results can be
found in \cite{Fa:Rigidity}.  See also \cite{KanRe:Ulam} for closely
related rigidity results in the Borel context (cf. \S\ref{S.Borel}
below).

\subsection{Terminology and Notation} \lbl{S.Notation} All the necessary
background on operator algebras can be found e.g., in \cite{Pede:Analysis} or
\cite{We:Set}.
 Throughout
we fix an infinite dimensional separable complex Hilbert space $H$
and an orthonormal basis $(e_n)$.  Let $\pi\colon \cB(H)\to \cC(H)$
be the quotient map.  If $F$ is a closed subspace of $H$ then
$\proj_F$ denotes the orthogonal projection to $F$. Fix an increasing
family of finite-dimensional projections $(\bfR_n)$ such that
$\bigvee_n \bfR_n=I$, and consider a nonincreasing family of
seminorms $ \|a\|_n=\|(I-\bfR_n)a\|$.  Let $\|a\|_{\cK}=\lim_{n\to
\infty} \|a\|_n$. Note that $\|a\|_{\cK}=\|\pi(a)\|$, with the norm
of $\pi(a)$ computed in the Calkin algebra. Projections $P$ and $Q$
are \emph{almost orthogonal} if $PQ$ is compact. This is equivalent
to $QP=(PQ)^*$ being compact.

Let $\cA$, $\cB$ be C*-algebras,   $J_1,J_2$ their ideals and let $\Phi\colon
\cA/J_1\to \cB/J_2$ be a
*-homomorphism. A map $\Psi\colon \cA\to \cB$ such that
  ($\pi_{J_i}$ is the quotient map)
$$
\diagram
\cA \xto[r]^{\Psi}\xto[d]^{\pi_{J_1}} &   \cB\xto[d]^{\pi_{J_2}}\\
\cA/J_1 \xto[r]_{\Phi} &\cB/J_2
\enddiagram
$$
commutes,  is a   \emph{representation}   of $\Phi$. Since we do  not require
$\Psi$ to  be a *-homomorphism, the Axiom of Choice implies every $\Phi$ has a
representation.

For a partition $\vec E$ of $\bbN$ into finite intervals $(E_n)$ let
$\calD[\vec E]$ be the von Neumann algebra of all operators in $\cB(H)$  for
which each $\SPAN\{e_i\mid i\in E_n\}$ is invariant. We always assume $E_n$ are
consecutive, so that $\max(E_n)+1=\min(E_{n+1})$ for each $n$. If $E_n=\{n\}$
then $\calD[\vec E]$ is the \emph{standard atomic masa}: von Neumann algebra of
all operators diagonalized by the standard basis. These FDD (short for `finite
dimensional decomposition') von Neumann algebras play an important role in the
proof of Theorem~\ref{T.0}.
 For $M\subseteq \bbN$ let
$\bP_M^{\vec E}$ (or $\bP_M$ if $\vec E$ is clear from the context) be the
projection to the closed linear span of $\bigcup_{i\in M}\{e_n\mid n\in E_i\}$
and let $\calD_M[\vec E]$ be the ideal $\bP_M\calD[\vec E]\bP_M=\bP_M\calD[\vec
E]$ of $\calD[\vec E]$. It is not difficult to see that an operator $a$ in
$\calD[\vec E]$ is compact if and only if $\lim_i \|\bP^{(\vec
E)}_{\{i\}}a\|=0$.
 The strong operator topology coincides with the product of the norm topology
on the unitary group of $\calD[\vec E]$, $\cU[\vec E]=\prod_i
\cU(E_i)$ and makes it into a compact metric group.

If $\cA$ is a unital C*-algebra then $\cU(\cA)$ denotes its unitary
group. We shall write $\cU[\vec E]$ for $\cU(\calD[\vec E])$ and
$\cU_A[\vec E]$ for $\cU(\calD_A[\vec E])$. Similarly, we shall write
$\cC[\vec E]$ for $\calD[\vec E]/(\calD[\vec E]\cap \cK(H))$. For a
$C^*$-algebra $\calD$ and $r<\infty$ write
\[
\calD_{\leq r}=\{a\in \calD\mid \|a\|\leq r\}.
\]
 The set of
self-adjoint operators in $\calD$ is denoted by $\calD_{\SA}$.

The \emph{spectrum} of a normal operator $b$ in a unital C*-algebra is
$$
\sigma(b)=\{\lambda\in \bbC\mid b-\lambda I\text{ is not invertible}\}.
$$

\subsubsection*{A rough outline of the proof of Theorem~\ref{T.0}}
If $\calD$ is a subset of $\cB(H)$, we say that $\Phi$ is \emph{inner on
$\calD$} if there is an inner automorphism $\Phi'$ of $\cC(H)$ such that the
restrictions of $\Phi$ and $\Phi'$ to $\pi[\calD]$ coincide. In
Theorem~\ref{T.1} we use CH to construct an outer automorphism of the Calkin
algebra whose restriction to each $\calD[\vec E]$ is inner.
 In Theorem~\ref{P.1} we use TA to show that for any outer automorphism
$\Phi$ there is $\vec E$ such that $\Phi$ is not inner on
 $\calD[\vec E]$. Both of these proofs involve the analysis of `coherent families of
unitaries' (\S\ref{S.CH}).

Fix an automorphism of the Calkin algebra $\Phi$. Fix $\vec E$ such that the
sequence $\# E_n$ is nondecreasing. A simple fact that $\Phi$ is inner on
$\calD[\vec E]$ if and only if it is inner on $\calD_M[\vec E]$ for some
infinite $M$ is given in Lemma~\ref{L.trivial.2}. Hence we only need to find an
infinite $M$ such that the restriction of $\Phi$ to $\calD_M[\vec E]$ is inner.
 This is done in Proposition~\ref{P.2}. Its proof proceeds in several stages and it involves
the notion of an $\e$-approximation (with respect to $\|\cdot
\|_{\cK}$) to a representation (see \S\ref{S.Representation}) and the
family $ \cJ^n(\vec E)=\{A\subseteq \bbN\mid \Phi$ has a C-measurable
$2^{-n}$-approximation on $\calD[\vec E]\}$.  In Lemma~\ref{L.cJ-m}
TA is used again to prove that $\cJ^n(\vec E)$ is so large for every
$n$ that  $\bigcap_n \cJ^n(\vec E)$ contains an infinite set $M$.
Jankov, von Neumann uniformization theorem (Theorem~\ref{T.JvN}) is
used to produce a C-measurable representation of~$\Phi$ on
$\calD_M[\vec E]$ as a `limit' of given $2^{-n}$-approximations. This
C-measurable representation is turned into a conjugation by a unitary
in Theorem~\ref{T.2}. This result depends on the Ulam-stability of
approximate
*-homomorphisms (Theorem~\ref{T.approximate}).

 Part of the present proof that deals with FDD von Neumann algebras
owes much to the proof of the `main lifting theorem' from
\cite{Fa:AQ} and a number of elegant improvements from  Fremlin's
account \cite{Fr:Notes}. In particular, the proof of
Claim~\ref{C.stabilizer.0} is based on the proof of
\cite[Lemma~1P]{Fr:Notes} and   \S\ref{S.sigma}  closely follows
\cite[Lemma 3C]{Fr:Notes}.


\section{An outer automorphism from the Continuum Hypothesis}\lbl{S.CH}
We first prove a slight strengthening of the Phillips--Weaver result.
Lem\-ma~\ref{L.Arv}, Lemma~\ref{L.1}, definitions of $\rho$ and $\Delta_I$, and
Lemma~\ref{L.A.4} will be needed in the proof of Theorem~\ref{T.0}.

\begin{thm}\lbl{T.1} The Continuum
Hypothesis implies there is an outer automorphism of the Calkin
algebra. Moreover, the restriction of this automorphism to the
standard atomic masa and to any separable subalgebra is inner.
\end{thm}

 If $\vec E$ and
$\vec F$ are partitions of $\bbN$ into finite intervals we
write $\vec E\leq^* \vec F$ if for all but finitely many
$i$ there is $j$ such that $E_i\cup E_{i+1}\subseteq
F_j\cup F_{j+1}$. A family~$\cE$ of partitions is
\emph{cofinal} if for every $\vec F$ there is $\vec E\in
\cE$ such that $\vec F\leq^* \vec E$.

Let ${\bbT}$ denote the circle group, $\{z\in \bbC\mid |z|=1\}$,   
and let $\cUN$ be its countable power. It is
isomorphic to the unitary group of the standard atomic masa. For
$\alpha\in\cUN$ let $u_\alpha$ be the unitary operator on $H$ that sends $e_n$
to $\alpha(n)e_n$. For a unitary $u$ let $\Psi_u$ be the conjugation by $u$,
$\Psi_u(a)=uau^*$ (usually denoted by $\Ad u$ in the operator algebras literature.) 
If $u=u_\alpha$ we write $\Psi_\alpha$ for
$\Psi_{u_\alpha}$. We say that $\Psi_\alpha$ and $\Psi_\beta$ \emph{agree
modulo compacts} on $\calD$ if $\Psi_\alpha(a)-\Psi_\beta(a)$ is compact for
every $a\in \calD$.

Given  $\vec E$ define two coarser partitions: $\vec E^{\even}$, whose entries
are $E_{2n}\cup E_{2n+1}$ and $\vec E^{\odd}$, whose entries are $E_{2n-1}\cup
E_{2n}$ (with $E_{-1}=\emptyset$). Let
$$
\cF[\vec E]=\calD[\vec E^{\even}]\cup \calD[\vec E^{\odd}].
$$
I have proved Lemma~\ref{L.Arv} below using the methods of \cite{Arv:Notes}.
George  Elliott pointed out that the proof of this lemma (in a more general
setting) is contained in the proof of
 \cite[Theorem~3.1]{Ell:Derivations}, as remarked in  \cite{Ell:Ideal}.

\begin{lemma}\lbl{L.Arv}
For a sequence $(a_n)$ in $\cB(H)$ there are a partition $\vec E$,  $a_n^0\in
\calD[\vec E^{\even}]$ and $a_n^1\in \calD[\vec E^{\odd}]$ such that
$a_n-a_n^0-a_n^1$ is compact for each~$n$.
\end{lemma}

\begin{proof}  For $A\subseteq \bbN$ write $\bP^{(e_n)}_A$ for
the projection to the closed linear span of $\{e_i\mid i\in A\}$. Fix
$m\in \bbN$ and $\e>0$. Since $a\bP_{[0,m)}$ is compact, we can find
$n>m$ large enough to have $\|\bP_{[n,\infty)}a\bP_{[0,m)}\|<\e$ and
similarly $\|\bP_{[n,\infty)}a^*\bP_{[0,m)}\|<\e$. Therefore
$\|\bP_{[0,m)}a\bP_{[n,\infty)}\|<\e$ as well. Recursively find a
strictly increasing $f\colon \bbN\to\bbN$ such that for all $m\leq n$
and  $i\leq n$ we have $\|\bP_{[f(n+1),\infty)}a_i\bP_{[0,f(m))}\|<2^{-n}$
and $\|\bP_{[0,f(m))}a_i\bP_{[f(n+1),\infty)}\|<2^{-n}$. We shall
check that $\vec E$ defined by $E_n=[f(n),f(n+1))$ is as required.
Write $Q_n=\bP_{[f(n),f(n+1))}$ (with $f(0)=0$). Fix $a=a_i$ and
define
\begin{align*}
a^0&=\textstyle\sum_{n=0}^\infty
(Q_{2n}aQ_{2n}+Q_{2n}aQ_{2n+1}+Q_{2n+1}aQ_{2n})\\
a^1&=\textstyle\sum_{n=0}^\infty
(Q_{2n+1}aQ_{2n+1}+Q_{2n+1}aQ_{2n+2}+Q_{2n+2}aQ_{2n+1}).
\end{align*}
Then $a^0\in \calD[\vec E^{\even}]$, $a^1\in \calD[\vec E^{\odd}]$. Let
$c=a-a^0-a^1$. For every $n$ we have 
\begin{align*}
\textstyle\|\bP_{[f(n),\infty)}c\|&\leq
\left\|\textstyle\sum_{i=n}^\infty \bP_{[f(i),\infty)} a
\bP_{[0,f(i-1))}\right\|+\left\|\textstyle\sum_{i=n+1}^\infty
\bP_{[f(n),f(i))}a\bP_{[f(i+1),\infty)}\right\|\\
& \leq 2^{-n+2}+2^{-n+1}, 
\end{align*}
and therefore
 $c$ is compact.
\end{proof}

 Whenever possible we collapse the subscripts/superscripts and write e.g., $\Psi_\xi$ for
$\Psi_{\alpha^\xi}$ (which is of course
$\Psi_{u_{\alpha^\xi}}$).

\begin{lemma} \lbl{L.1} Assume $(\vec E^\xi)_{\xi\in \Lambda}$  is a
directed cofinal family of partitions and~$\alpha^\xi$, $\xi\in \Lambda$, are
such that $\Psi_\eta$ and $\Psi_\xi$ agree modulo compacts on $\cF[\vec E^\xi]$
for $\xi\leq \eta$. Then there is an automorphism $\Phi$ of $\cC(H)$ such that
$\Psi_\xi$ is a representation of $\Phi$ on $\cF[\vec E^\xi]$ for every $\xi\in
\Lambda$. Moreover, $\Phi$ is unique.
\end{lemma}

\begin{proof}  By Lemma~\ref{L.Arv}, for each $a\in \cB(H)$
there is a partition $\vec E$ with  $a_0\in \calD[\vec E^{\even}]$ and  $a_1\in
\calD[\vec E^{\odd}]$ such that $a-a_0-a_1$ is compact. Fix  $\vec F=\vec E^\xi$
such that $\vec E\leq^* \vec F$ and  let $\Phi(\pi(a))=\pi(\Psi_{\vec F}(a))$.

Then $\Phi$ is well-defined by the agreement of $\Psi_\xi$'s. For every pair of
operators $a,b$ there is a single partition $\vec E$ with $a_0$ and $b_0$ in
$\calD[\vec E^{\even}]$ and $a_1$ and $b_1$ in $\calD[\vec E^{\odd}]$ such that
both $a-a_0-a_1$ and $b-b_0-b_1$ are compact. This readily implies $\Phi$ is a
*-homomorphism.

The inverse maps $\Psi_\xi^*=\Psi_{(u_{\alpha^\xi})^*}$ also satisfy the
assumptions of the lemma and there is a
*-homomorphism $\Phi^*$ such that $\Psi_\xi^*$ is a representation
of $\Phi^*$ on $\cF[\vec E^\xi]$ for every $\xi$. Then
 $\Phi\Phi^*=\Phi^*\Phi$ is the identity on
$\cC(H)$, hence $\Phi$ is an automorphism. The uniqueness follows from
Lemma~\ref{L.Arv}.
\end{proof}

 Let ${\bbT}$ be the
unitary group of the 1-dimensional complex Hilbert space.  Recall that every
inner automorphism of $\cC(H)$ has a representation of the form $\Psi_u$ for
$u$ which is  an isometry between subspaces of $H$ of finite codimension.  The
proof of the following lemma was suggested by Nik Weaver.

\begin{lemma} \lbl{L.A.3}
Assume  $u$ and $v$ are isometries between subspaces of $H$ of finite
codimension. If $\Psi_u(a)-\Psi_v(a)$ is compact for every $a$
diagonalized by $(e_n)$, then there is $\alpha\in ({\bbT})^{\bbN}$
for which the linear map $w$ defined by $w(e_n)=\alpha(n)v(e_n)$ for
all $n$ is such that $\Psi_w(a)-\Psi_u(a)$ is compact for all $a$ in
$\cB(H)$.
\end{lemma}

\begin{proof}
Let $\cA=\{a\in \cB(H)\mid a$ is diagonalized by $(e_n)\}$. By our
assumption, $\pi(v^*u)$ commutes with $\pi(a)$ for all $a\in \cA$.
 Since $\pi[\cA]$ is, by \cite{JohPar},  a maximal abelian self-adjoint
subalgebra of the Calkin algebra we have  $\pi(w_0)=\pi(v^*u)$ for
some $w_0\in \cA$. Let $w_0=bw_1$ be the polar decomposition of $w_0$
in $\cA$. Since $\pi(b)=I$ we have $\pi(w_1)=\pi(v^*u)$. Since the
Fredholm index of $w_1$ is 0 and $\pi(w_1)$ is a unitary, we may
assume $w_1$ is a unitary. Fix $\alpha\in \cUN$ such that
$w_1=u_{\alpha}$, i.e., $w_1(e_n)=\alpha(n)e_n$ for all $n$.

Let $w=vw_1$. Then $\pi(w)=\pi(vv^*u)=\pi(u)$ hence
$\Psi_w(a)-\Psi_u(a)$ is compact for all $a\in \cB(H)$. Also, for
each $n$ we have $w(e_n)=vw_1(e_n)=v(\alpha(n)e_n)=\alpha(n)v(e_n)$.
\end{proof}

For $i,j$ in $\bbN$ and $\alpha,\beta$ in $\cUN$  let
\[
\rho(i,j,\alpha,\beta)=
|\alpha(i)\overline{\alpha(j)}-\beta(i)\overline{\beta(j)}|.
\]
 For
fixed $i,j$ the function $f\equiv \rho(i,j,\cdot,\cdot)$
satisfies the triangle inequality:
\[
f(\alpha,\beta)+f(\beta,\gamma)\geq f(\alpha,\gamma).
\]
We also have
\[
\rho(i,j,\alpha,\beta)=|\rho(i,j,\alpha,\beta)\alpha(j)\overline{\beta(i)}|=|\alpha(i)\overline{\beta(i)}-\alpha(j)\overline{\beta(j)}|,
\]
hence  for fixed $\alpha,\beta$ the function $f_1\equiv
\rho(\cdot,\cdot,\alpha,\beta)$ also satisfies the triangle
inequality:
\[
f_1(i,j)+f_1(j,k)\geq f_1(i,k).
\]
 For $I\subseteq \bbN$ and $\alpha$ and $\beta$ in
$\cUN$ write
\[
 \Delta_I(\alpha,\beta)=\sup_{i\in I,\  j\in I}
\rho(i,j,\alpha,\beta).
\]
The best picture $\Delta_I$ is  furnished by \eqref{L.A.5.5.new} of the following lemma. 

\begin{lemma}\lbl{L.A.5.5}
For all $I,\alpha,\beta$ we have
\begin{enumerate}
\item  $\Delta_I(\alpha,\beta)\leq
2\sup_{i\in I}|\alpha(i)-\beta(i)|$.
\item
 $\Delta_I(\alpha,\beta)\geq\sup_{j\in I}|\alpha(j)-\beta(j)|-\inf_{i\in
 I}|\alpha(i)-\beta(i)|$, in particular if
 $\alpha(i_0)=\beta(i_0)$ for some $i_0\in I$ then
 $\Delta_I(\alpha,\beta)\geq \sup_{j\in
 I}|\alpha(j)-\beta(j)|$.
\item If $z\in {\bbT}$ then
$\Delta_I(\alpha,\beta)=\Delta_I(\alpha,z\beta)$.
\item If $I\cap J$ is nonempty then $\Delta_{I\cup
J}(\alpha,\beta)\leq
\Delta_I(\alpha,\beta)+\Delta_J(\alpha,\beta)$.
\item \lbl{L.A.5.5.new}
$
\inf_{z\in \bbT}\sup_{i\in I} |\alpha(i)-z\beta(i)|\leq  
\Delta_I(\alpha,\beta)\leq 2
\inf_{z\in \bbT}\sup_{i\in I} |\alpha(i)-z\beta(i)|$. 
\end{enumerate}
 \end{lemma}

\begin{proof} 
Since
\begin{align*}
\rho(i,j,\alpha,\beta)=|\alpha(i)\overline{\beta(i)}-\alpha(j)\overline{\beta(j)}|
&=|\overline{\beta(i)}
(\alpha(i)-\beta(i))+\overline{\beta(j)}(\beta(j)-\alpha(j))|
\end{align*}
and $|\beta(i)|=|\beta(j)|=1$, we have
\[
 \left||\alpha(i)-\beta(i)|-|\alpha(j)-\beta(j)|\right|\leq
\rho(i,j,\alpha,\beta)\leq |\alpha(i)-\beta(i)|+|\alpha(j)-\beta(j)|.
\]
This implies
 \begin{align*}
 \Delta_I(\alpha,\beta)&=\sup_{i\in I,\ j\in I}
 \rho(i,j,\alpha,\beta)\\
 &\leq \sup_{i\in I,\ j\in
 I}(|\alpha(i)-\beta(i)|+|\alpha(j)-\beta(j)|)
 \leq 2\sup_{i\in I} |\alpha(i)-\beta(i)|
 \end{align*}
and (1) follows. For (2) we have
\begin{align*}
\Delta_I(\alpha,\beta)&=\sup_{i\in I,\ j\in
I}\rho(i,j,\alpha,\beta)\\
&\geq \sup_{i\in I,\ j\in I}
||\alpha(i)-\beta(i)|-|\alpha(j)-\beta(j)||\\
&=\sup_{i\in I}|\alpha(i)-\beta(i)|-\inf_{i\in
I}|\alpha(i)-\beta(i)|.
\end{align*}
Clause (3) is an immediate consequence of the equality
$\rho(i,j,\alpha,z\beta)=\rho(i,j,\alpha,\beta)$. It is not difficult
to see that in order to prove (4), we only need to check
$\rho(i,j,\alpha,\beta)\leq
\Delta_I(\alpha,\beta)+\Delta_J(\alpha,\beta)$ for all $i\in I$ and
$j\in J$. Pick $k\in I\cap J$. Then we have
\[
\rho(i,j,\alpha,\beta)\leq
\rho(i,k,\alpha,\beta)+\rho(k,j,\alpha,\beta) \leq
\Delta_I(\alpha,\beta)+\Delta_J(\alpha,\beta),
\]
completing the proof.

Now we prove \eqref{L.A.5.5.new}. By the definition, for every $j\in I$ we have 
$\Delta_I(\alpha,\beta)\geq \sup_{i\in I} |\alpha(i)-(\alpha(j)\overline{\beta(j)})\beta(i)|$
and therefore $\Delta_I(\alpha,\beta)\geq \inf_{z\in \bbT} \sup_{i\in I}|\alpha(i)-\beta(i)|$. 
On the other hand, for all $i$ and $j$ we have 
\[
|\alpha(i)-(\alpha(j)\overline{\beta(j)})\beta(i)|\leq 
|\alpha(i)-z\beta(i)|+|\alpha(j)-z\beta(j)|
\]
 for every $z\in \bbT$, 
which immediately implies the other inequality. 
\end{proof}

Recall that for $\alpha\in \cUN$ by $u_\alpha$ we denote the unitary
such that $u_\alpha(e_n)=\alpha(n)e_n$ and that
$\Psi_\alpha=\Psi_{u_\alpha}$ is the conjugation by $u_\alpha$.

\begin{lemma} \lbl{L.A.4}
\begin{enumerate}
\item [(a)] If  $\lim_n |\alpha(n)-\beta(n)|=0$ then  $\Psi_\alpha(a)-\Psi_\beta(a)$ is compact for all $a\in
\cB(H)$.
\item [(b)] The difference $\Psi_\alpha(a)-\Psi_\beta(a)$ is compact for all $a\in
\calD[\vec E]$ if and only if $\limsup_n \Delta_{E_n}(\alpha,\beta)=0$.
\end{enumerate}
\end{lemma}

\begin{proof}
(a) Since $\lim_n |\alpha(n)-\beta(n)|=0$ implies $\pi(u_\alpha)=\pi(u_\beta)$, 
we have that 
$\Psi_\alpha(a)-\Psi_\beta(a)=(u_\alpha-u_\beta)a(u_\alpha^*-u_\beta^*)$ is compact. 

(b)  Assume $\limsup_n \Delta_{E_n}(\alpha,\beta)=0$. For each $n$ let
$m_n=\min (E_n)$ and define $\gamma\in \cUN$ by
$$
\gamma(i)=\beta(i)\overline{\beta(m_n)}\alpha(m_n), \quad \text{if $i\in E_n$.}
$$
The operator $\sum_{n\in \bbN}
\overline{\beta(m_n)}\alpha(m_n)\proj_{E_n}$ (with the obvious
interpretation of the infinite sum) is central in $\calD[\vec E]$ and
therefore for every $a\in \calD[\vec E]$ we have
$\Psi_\gamma(a)=\Psi_\beta(a)$. By clauses (2) and (3) of
Lemma~\ref{L.A.5.5} we have  $|\gamma(i)-\alpha(i)| \leq
\Delta_{E_n}(\alpha,\gamma)= \Delta_{E_n}(\alpha,\beta)$ for $i\in
E_n$. Therefore $\lim_i |\gamma(i)-\alpha(i)|=0$ and the conclusion
follows by (a).

Now assume $\limsup_n \Delta_{E_n}(\alpha,\beta)>0$. Fix $\e>0$, an
increasing sequence~$n(k)$ and $i(k)<j(k)$ in $E_{n(k)}$ such that
$\rho(i(k),j(k),\alpha,\beta)\geq \e$ for all~$k$. The partial
isometry  $a$ defined by $a(e_{i(k)})=e_{j(k)}$,
$a(e_{j(k)})=e_{i(k)}$, and $a(e_j)=0$ for other values of $j\in
E_{n(k)}$ belongs to $\calD[\vec E]$. 
Then
\begin{multline*}
\Psi_\alpha(a)(e_{i(k)})=(u_\alpha a u_\alpha^*)(e_{i(k)})=
u_\alpha(a(\overline{\alpha(i(k))} e_{i(k)})\\
=u_\alpha(\overline{\alpha(i(k))} e_{j(k)})
=\alpha(j(k))\overline{\alpha(i(k))} e_{j(k)}.
\end{multline*}
Similarly
$\Psi_\beta(a)(e_{i(k)})=\beta(j(k))\overline{\beta(i(k))}e_{j(k)}$
for all $k$. Therefore
\[
\|(\Psi_\alpha(a)-\Psi_\beta(a))(e_{i(k)})\|\geq
\rho(i(k),j(k),\alpha,\beta)\geq \e
\]
 for all $k$,
and the difference $\Psi_\alpha(a)-\Psi_\beta(a)$ is not compact.
\end{proof}

\begin{proof}[Proof of Theorem~\ref{T.1}]
Enumerate $\cUN$ as $\beta^\xi$ for $\xi<\omega_1$ and all partitions
of~$\bbN$ into finite intervals as $\vec F^\xi$, with $\xi<\omega_1$.
 Construct a
$\leq^*$-increasing cofinal chain $\vec E^\xi$ of
partitions and $\alpha^\xi\in \cUN$ such that for all
$\xi<\eta$ we have
\begin{enumerate}
\item \lbl{I.1}$\limsup_n\Delta_{E^\xi_n\cup
E^\xi_{n+1}}(\alpha^\xi,\alpha^\eta)=0$.
\item \lbl{I.2}$\limsup_n
\Delta_{E^{\xi+1}_n}(\alpha^{\xi+1},\beta^\xi)\geq \sqrt
2$. 
\item   $E^\eta$
is eventually coarser than $E^\xi$ in the sense that $E^\eta_m$ 
is equal to a union of intervals from $E^\xi$ for all but finitely many $m$. 
\pushcounter
\end{enumerate}
In order to describe the recursive construction, we
consider two cases.

First, assume $\zeta<\omega_1$ and $\vec E^\xi$ and
$\alpha^\xi$ were chosen for all $\xi\leq \zeta$. Let $\vec
E^{\zeta+1}$ be such that $F_n=E^{\zeta+1}_n$ is the union
of $2n+1$ consecutive intervals of $\vec E^\zeta$, denoted
by  $F^n_0,\dots, F^n_{2n}$.  Fix $n$. If $\Delta_{\vec
E^{\zeta+1}_n}(\alpha^\zeta,\beta^\zeta)\geq \sqrt 2$ let
$\alpha^{\zeta+1}$ coincide with $\alpha^\zeta$ on
$E^{\zeta+1}_n$. Now assume
 $\Delta_{E^{\zeta+1}_n}(\alpha^\zeta,\beta^\zeta)<\sqrt 2$.
 Let $\gamma_n=\exp(i\pi/n)$. Let
$\alpha^{\zeta+1}(j)= \gamma_n^k\alpha^\zeta(j)$ for  $j\in
F^n_k$. If $i\in F^n_0$ and $j\in F^n_n$ then
$\alpha^{\zeta+1}(i)=\alpha^\zeta(i)$ and
$\alpha^{\zeta+1}(j)=-\alpha^\zeta(j)$.  Since
$|\alpha^\zeta(j)\overline{\alpha^\zeta(i)}-
\beta^\zeta(j)\overline{\beta^\zeta(i)}|< \sqrt 2$, we have
$\Delta_{E^{\zeta+1}_n}(\alpha^{\zeta+1},\beta^\zeta)
\geq|\alpha^\zeta(j)\overline{\alpha^\zeta(i)}+\beta^\zeta(j)\overline{\beta^\zeta(i)}|>\sqrt
2$.

Hence \eqref{I.2} holds. We need to check $\limsup_m
\Delta_{E^\zeta_m\cup
E^\zeta_{m+1}}(\alpha^\zeta,\alpha^{\zeta+1})=0$. We have
$\Delta_{E^\zeta_m}(\alpha^\zeta,\alpha^{\zeta+1})=0$ for
all $m$. Since
 $\alpha^{\zeta+1}$ and
$\alpha^\zeta$ coincide on $F^n_0$ and on $F^n_{2n}$ for
each $n$, $\Delta_{F^n_{2n}\cup F^{n+1}_0}
(\alpha^\zeta,\alpha^{\zeta+1})=0$ for all $n$. If $0\leq
k<2n$ then  $\Delta_{F^n_k\cup
F^n_{k+1}}(\alpha^\zeta,\alpha^{\zeta+1})\leq
|\gamma_n|\leq |\sin(\pi/n)|\leq \pi/n$. Hence clause
\eqref{I.1} is satisfied with $\xi=\zeta$ and
$\eta=\zeta+1$, and therefore it holds for all $\xi$ and
$\eta=\zeta+1$ by transitivity.

Now assume $\zeta<\omega_1$ is a limit ordinal such that $\vec E^\xi$
and $\alpha^\xi$ have been defined for $\xi<\zeta$. Let $\xi_n$, for
$n\in \bbN$,  be an increasing sequence with supremum $\zeta$ and
write $\vec E^n,\alpha^n$ for $\vec E^{\xi_n},\alpha^{\xi_n}$. Find a
strictly increasing function $f\colon \bbN\to \bbN$ such that
\begin{enumerate}
\popcounter
\item \lbl{J.1}$f(0)=0$,
\item \lbl{J.2b}
$f(n+1)$ is large enough so that each $E^{n+1}_i$ disjoint from $[0,f(n+1))$ is 
the union of finitely many intervals of $E^n$.  
\item  \lbl{J.2}For all  $k\leq n$ the interval 
$F_n=[f(n),f(n+1))$ is the union of finitely many intervals from $E^k$,
\item \lbl{J.3} If $l<k\leq n$, $j\in \bbN$,
and $\Delta_{E^l_j\cup E^l_{j+1}}(\alpha^l,\alpha^k)\geq 1/n$, then $\max E^l_j\leq f(n)$.
\item \lbl{J.4}  $F^\zeta_i\not\supseteq F_n$ for all $i$ and
all $n$.
\pushcounter
\end{enumerate}
The values $f(n)$ for $n\in \bbN$ are chosen recursively. If $f(n)$
has been chosen then each of the clauses \eqref{J.2b}, \eqref{J.2}, \eqref{J.3} and
\eqref{J.4} can be satisfied by choosing $f(n+1)$ to be larger than
the maximum of a finite subset of $\bbN$.

Assume $f$ has been chosen to satisfy \eqref{J.1}--\eqref{J.4}. By
\eqref{J.2} for $m$ and $i\geq m$ we have $E^m_i\cup
E^m_{i+1}\subseteq F_n\cup F_{n+1}$ if $n$ is the maximal such that
$f(n)<\min E^m_i$.
 Therefore with  $\vec E^\zeta=\vec F$ we have $\vec
E^n\leq^* \vec E^\zeta$ for all $n$, and therefore $\vec E^\xi\leq^*
\vec E^\zeta$ for all $\xi<\zeta$. By \eqref{J.4} we have that for
each $i\in\bbN$ the interval
 $F^\zeta_i$ intersects at most two of the intervals $F_n$ nontrivially and therefore
  $\vec F^\xi\leq^* \vec E^\xi$.

Recursively define $\gamma_n\in {\bbT}$ for $n\in \bbN$ and $\alpha^\zeta(j)$ for $j\in \bbN$
so that for all~$n$ we have 
\begin{enumerate}
\popcounter
\item\label{J.alpha} $\alpha^\zeta(j)=\gamma_n\alpha^n(j)$ for $j\in   F_n\cup \{f(n+1)\}$.
\end{enumerate}
To this end, let 
\begin{align*} 
\alpha^\zeta(j)&=\alpha^0(j)&\text{ for }j\in F_0\cup \{f(1)\}, \\
&\gamma_1=\alpha^0(f(1))\overline{\alpha^1(f(1))}\\
\alpha^\zeta(j)&=\gamma_1\alpha^1(j)&\text{ for }j\in F_1\cup \{f(2)\}, 
\end{align*}
and in general for $n\geq 0$ let  
\[
\gamma_{n+1}=\alpha_n(f(n+1))\gamma_n \overline{\alpha^{n+1}(f(n+1))}
\]
let 
\[
\alpha^\zeta(j)=\gamma_{n+1}\alpha^{n+1}(j)
\qquad\text{  for $j\in F_{n+1}\cup \{f(n+2)\}$}. 
\]
This sequence satisfies \eqref{J.alpha}. 

Fix $m$ and write $I_n$ for $E^m_n\cup E^m_{n+1}$. We want to show  
$
\lim_{n\to \infty} \Delta_{I_n}(\alpha^\zeta,\alpha^m)=0$. 

By 
\eqref{J.2}, 
for all $n\geq m$ we have $E^m_n\subseteq F_k$,  for some $k=k(n)$ and  
   therefore 
   $I_n\subseteq F_k\cup F_{k+1}$. This implies, 
using    Lemma~\ref{L.A.5.5}(4),(3), that 
   \[
\Delta_{I_n}(\alpha^\zeta,\alpha^m)\leq \Delta_{E^m_n\cup \{f(k+1)\}}(\alpha^k,\alpha^m)
+\Delta_{E^m_{n+1}}(\alpha^{k+1},\alpha^m)
\]
and by \eqref{J.3}
 the right-hand side is $\leq \frac 1k+\frac 1{k+1}$, 
since $n\geq m$. 
Moreover  $\lim_{n\to \infty} k(n)=\infty$, and we can conclude that 
$\lim_{n\to \infty} \Delta_{I_n}(\alpha^\zeta,\alpha^m)=0$. 
Therefore the conditions of Lemma~\ref{L.A.4} (b) are satisfied and
$\alpha^\zeta$ satisfies~\eqref{I.1}.

This finishes the description of the construction of $\vec E^\xi$ and
$\alpha^\xi$ satisfying \eqref{I.1} and \eqref{I.2}. By
Lemma~\ref{L.1} there is an automorphism $\Phi$ of $\cC(H)$ that 
has~$\Psi_\xi$ as its representation on $\cF[\vec E^\xi]$ for each $\xi$.
Assume this automorphism is inner. Then it has a representation of
the form $\Psi_u$ for some partial isometry $u$.  By
Lemma~\ref{L.A.3} applied to $\Psi_0$ and $\Psi_u$
there is $\beta\in \cUN$ such that $\Psi_\beta$ is
a representation of $\Phi$.  But $\beta$ is equal to $\beta^\xi$ for
some $\xi<\omega_1$, and by \eqref{I.2} and Lemma~\ref{L.A.4} (b) the
mapping $\Psi_\beta$ is not a representation of $\Phi$ on $\cF[\vec
E^\zeta]$.

By construction, the constructed automorphism is inner on the standard atomic
masa, and actually on each $\calD[\vec E]$. In addition, Lemma~\ref{L.Arv}
shows that for every countable subset of $\cC(H)$ there is an inner
automorphism of~$\cC(H)$ that sends $a$ to~$\Phi(a)$.
\end{proof}

In the proof of Theorem~\ref{T.1} CH was used only in the first line to find
enumerations $(\vec F^\xi)$ and $(\beta_\xi)$,  $\xi<\omega_1$. The first
enumeration was used to find a $\leq^*$-cofinal $\omega_1$-sequence of
partitions $\vec E^\xi$ and the second to assure that $\Phi\neq
\Psi_{\beta_\xi}$ for all $\xi$. A weakening of CH known as $\fd=\aleph_1$ (see
e.g., \cite{BarJu:Book}) suffices for the first task. Stefan Geschke pointed
out that the proof of Theorem~\ref{T.1} easily gives $2^{\aleph_1}$
automorphisms and therefore that the existence of the second enumeration may be
replaced with another cardinal inequality, $2^{\aleph_0}<2^{\aleph_1}$
(so-called  \emph{weak Continuum Hypothesis}). Therefore the assumptions
$\fd=\aleph_1$ and $2^{\aleph_0}<2^{\aleph_1}$ together imply the existence of
an outer automorphism of the Calkin algebra. It not known whether these
assumptions imply the existence of a nontrivial automorphism of
$\cP(\bbN)/\Fin$.

\section{The toolbox}
\lbl{S.toolbox}
\subsection{Descriptive set theory} The standard reference is
\cite{Ke:Classical}. A topological space is \emph{Polish} if it is separable
and completely metrizable.
 We consider $\cB(H)$ with the strong
operator topology. For every $M<\infty$ the strong operator topology on
$(\cB(H))_{\leq M}=\{a\in \cB(H)\mid \|a\|\leq M\}$ is Polish.
 Throughout `Borel' refers to the
Borel structure on $\cB(H)$ induced by the strong operator topology.

Fix  a Polish space $X$. A subset of $X$  is \emph{meager} (or, it is  of \emph{first category}) 
if it can be covered
by countably many closed nowhere dense sets.  A subset of $X$ has the
\emph{Property of Baire} (or, is \emph{Baire-measurable}) if its symmetric
difference with some open set is meager.
 A subset of  $X$ is
\emph{analytic} if it is a continuous image of a Borel subset of a Polish
space. Analytic sets (as well as their complements, \emph{coanalytic sets}),
share the classical regularity properties of Borel sets such as the Property of
Baire and measurability with respect to Borel measures. A function $f$ between
Polish spaces is \emph{C-measurable} if it is measurable with respect to the
smallest $\sigma$-algebra generated by analytic sets. C-measurable functions
are Baire-measurable (and therefore continuous on a dense $G_\delta$ subset of
the domain) and, if the domain is also a locally compact topological group,
Haar-measurable. The following uniformization theorem will be used a large
number of times in the proof of Theorem~\ref{T.0};  for its proof see e.g.
\cite[Theorem~18.1]{Ke:Classical}.

\begin{thm}[Jankov, von Neumann] \lbl{T.JvN}
If $X$ and $Y$ are Polish spaces and $A\subseteq X\times Y$
is analytic, then there is a C-measurable function $f\colon
X\to Y$ such that for all $x\in X$, if $(x,y)\in A$ for
some $y$ then $(x,f(x))\in A$. \qed
\end{thm}

 A function $f$ as above \emph{uniformizes}
$A$. In general it is impossible to uniformize a Borel set
by a Borel-measurable function, but the following two 
 special cases of
\cite[Theorem~8.6]{Ke:Classical} (applied with ${\mathcal I}_x$
being the meager ideal or the null ideal, respectively,   for each $x\in X$)
will suffice for our purposes.

\begin{thm}\lbl{T.U.Large}
Assume $X$  and $Y$ are Polish spaces,  $A\subseteq X\times Y$ is Borel and for
each $x\in X$ the vertical section $A_x=\{y\mid (x,y)\in A\}$ is either empty
or nonmeager. Then $A$ can be uniformized by a Borel-measurable function. \qed
\end{thm}

\begin{thm}\lbl{T.U.Large.Measure}
Assume $X$  and $Y$ are Polish spaces,  $Y$ carries a  Borel probability 
measure,  $A\subseteq X\times Y$ is Borel and for
each $x\in X$ the vertical section $A_x=\{y\mid (x,y)\in A\}$ is either empty
or has a positive measure. Then $A$ can be uniformized by a Borel-measurable function. \qed
\end{thm}

A topological group is \emph{Polish} if has a compatible complete separable
metric. The unitary group of $\cB(H)$, for a separable $H$, is a Polish group
with respect to the strong operator topology (e.g.,
\cite[9.B(6)]{Ke:Classical}). Also,
 the unitary group of every separably acting von Neumann algebra $\calD$ is
Polish with respect to strong operator topology. A complete separable metric on
$\cU(\calD)$ is given by $d'(a,b)=d(a,b)+d(a^*,b^*)$, where $d$ is the usual
complete metric on $\calD_{\leq 1}$ compatible with the strong operator
topology.  The following is Pettis's theorem (e.g.,
\cite[Theorem~9.10]{Ke:Classical}).

\begin{thm}\lbl{T.Pettis} Every Baire-measurable
homomorphism from a Polish group into a second-countable
group is continuous. \qed
\end{thm}

We end this subsection with a simple computation.

\begin{lemma} \lbl{L.T.3}
Consider $\cB(H)$ with the strong operator topology. Fix $M<\infty$.
\begin{enumerate}
\item [(a)] The set  of compact operators of
norm $\leq M$ is a Borel subset of $\cB(H)_{\leq M}$.
\item [(b)] For $\e\geq 0$ the set of operators $a$ of norm $\leq M$
such that $\|a\|_{\cK}\leq \e$ is a Borel subset of $\cB(H)_{\leq M}$.
\end{enumerate}
\end{lemma}

\begin{proof} (a)  Recall that $\bfR_n$ is a fixed increasing sequence of
finite-rank projections such that $\bigvee_n \bfR_n=I$. For a projection $P$
the set $\{a\mid \|Pa\|\leq x\}$ is strongly closed for every $x\geq 0$, and
\[
\cK(H)=\{a\mid (\forall m)(\exists n) \|(I-\bfR_n)a\|<1/m\}.
\]
 Hence $\cK(H)\cap
\cB(H)_{\leq M}$ is a relatively $F_{\sigma\delta}$ subset of $\cB(H)_{\leq M}$
for each~$M$.

The proof of (b) is almost identical.
\end{proof}

\subsection{Set theory of the power-set of the natural numbers}
\lbl{S.P(N)}
  A metric~$d$ on  $\p(\bbN)$ is defined by
$d(A,B)=2^{-\min(A\Delta B)}$, where $A\Delta B$ is the symmetric difference of
$A$ and $B$. This turns $(\p(\bbN),\Delta)$ into a compact metric topological
group, and the natural identification of subsets of $\bbN$ with infinite
sequences of zeros and ones is a homeomorphism into the triadic Cantor set.

\subsection{Todorcevic's axiom}\lbl{S.OCA}  Let    $X$ be a separable metric
space and let
$$
[X]^2=\{\{x,y\}\vert x\neq y\text{ and } x,y\in X\}.
$$
Subsets of $[X]^2$  are naturally identified with the symmetric
subsets of $X\times X$ minus the diagonal. A \emph{coloring}
$[X]^2=K_0\cup K_1$ is \emph{open} if $K_0$, when identified with a
symmetric subset of $X\times X$, is open in the product topology. If
$K\subseteq [X]^2$ then a subset $Y$ of $X$ is \emph{$K$-homogeneous}
if $[Y]^2\subseteq K$. Since  $K_1=[X]^2\setminus K_0$ is closed, a
closure of a $K_1$-homogeneous set is always $K_1$-homogeneous. The
following axiom was  introduced by Todorcevic in \cite{To:Partition}
under the name of Open Coloring Axiom, OCA.

\begin{TAdef}
If $X$ is a separable metric space  and $[X]^2=K_0\cup K_1$
is an open coloring, then $X$ either has  an uncountable
$K_0$-homogeneous subset or it  can be covered by a
countable family of $K_1$-homogeneous sets.
\end{TAdef}

The instance of TA when $X$ is analytic  follows from the usual
axioms of mathematics (see e.g., \cite{Feng:OCA}). In this case the
uncountable $K_0$-homogeneous set can be chosen to be perfect, hence
this variant of TA is a generalization of the classical perfect-set
property for analytic sets (\cite{Ke:Classical}).

 Note that $K_1$ is not required
to be open. We should say a word to clarify our use of the phrase
`open coloring.' In order to be able to apply TA to some coloring
$[X]^2=K_0\cup K_1$ it suffices to know that there is a separable
metric topology $\tau$ on $X$ which makes $K_0$ open. For example,
for  $X\subseteq \p(\bbN)$ and for each $x\in X$ we fix  an $f_x\in
\bbN^\bbN$ consider the coloring $[X]^2=K_0\cup K_1$ defined by
\begin{enumerate}
\item [ ]  $\{x,y\}\in K_0$ if and only if $f_x(n)\neq f_y(n)$ for some
$n\in x\cap y$.
\end{enumerate}
This  $K_0$ is not necessarily open in the topology inherited from $\p(\bbN)$
(\S\ref{S.P(N)}). However,  it is open in the topology obtained by identifying
$X$ with a subspace of $\p(\bbN)\times\bbN^\bbN$ via the embedding $x\mapsto
\<x,f_x\>$. We shall use such refinements tacitly quite often and  say only
that the coloring $[X]^2=K_0\cup K_1$ is open, meaning that it is open in some
separable metric topology.

Assume  $K_0\subseteq [X]^2$ is equal to a  union of
countably many rectangles, $K_0=\bigcup_i U_i\times V_i$.
If sets $U_i$ and $V_i$ separate points of $C$, then this
is equivalent to $K_0$
 being open in some separable
metric topology on $X$. Even without this assumption, by
\cite[Proposition~2.2.11]{Fa:AQ}, TA is equivalent to its apparent
strengthening to colorings $K_0$ that can be expressed as a union of
countably many rectangles. Reformulations of TA are  discussed at
length in \cite[\S2]{Fa:AQ}.

\subsection{Absoluteness} A
\emph{vertical section} of $B\subseteq X\times Y$ is a set of the form
$B_x=\{y\mid (x,y)\in B\}$ for some $x\in X$.

\begin{thm}
\lbl{T.Shoenfield} Assume $X$ and $Y$ are Polish spaces and $B\subseteq X\times
Y$ is Borel. Truth (or falsity) of the assertion that some vertical section of
$B$ is empty cannot be changed by going to a forcing extension.

In particular, if there is a proof using TA that $B$ has an empty
vertical section, than $B$ has an empty vertical section.
\end{thm}

\begin{proof} The first part is a special case of Shoenfield's Absoluteness
Theorem (see e.g., \cite[Theorem~13.15]{Kana:Book}). The second part
follows from a fact that every model of ZFC has a forcing extension
in which TA holds (\cite{Ve:Applications}).
\end{proof}

\section{Coherent families of unitaries} \lbl{S.CFU}

If $u$ is a partial isometry between cofinite-dimensional subspaces of $H$ we
write $\Psi_u(a)=uau^*$. An operator in $\cC(H)$ is invertible if and only if
it is of the form $\pi(a)$ for some Fredholm operator  $a$ (this is Atkinson's
theorem, \cite[Theorem 3.11.11]{Pede:Analysis}; see also
\cite[\S3]{BrDoFi:Unitary}). Therefore, inner automorphisms of $\cC(H)$ are
exactly the ones of the form $\Psi_u$ for a partial isomorphism $u$ between
cofinite-dimensional subspaces of $H$.
 A family $\cF$ of pairs $(\vec E,u)$ such
that
\begin{enumerate}
\item \lbl{I.coherent.1} If $(\vec E,u)\in \cF$ then $\vec E$ is a partition of $\bbN$ into finite intervals
and $u$ is a partial isometry between cofinite-dimensional
suspaces of $H$,
\item  for all $(\vec E,u)$
and $(\vec F,v)$ in $\cF$ and all $a\in \calD[\vec E]\cap \calD[\vec
F]$ the operator $\Psi_u(a)-\Psi_v(a)$ is compact,
\item for every partition $\vec E$ of $\bbN$ into finite intervals
there is $u$ such that $(\vec E,u)\in \cF$,
\end{enumerate}
 is called a \emph{coherent family of unitaries}. (By \eqref{I.coherent.1} above,
 $\pi(u)$ is a unitary in the Calkin algebra for each
 $(\vec E,u)\in \cF$.)
 The
 following is an immediate consequence of
 Lemma~\ref{L.A.3} and Lemma~\ref{L.1}.

\begin{lemma}\lbl{L.U.1} If $\cF$ is a coherent family of
unitaries, then there is a unique automorphism $\Phi$ of $\cC(H)$ such that
$\Psi_u$ is a representation of $\Phi$ on $\calD[\vec E]$ for all $(\vec
E,u)\in \cF$. \qed
\end{lemma}

Such an automorphism  $\Phi$ is \emph{determined by a coherent family of
unitaries}. Since $\calD[\vec E]\subseteq \calD[\vec F]$ whenever $\vec F$ is
coarser than $\vec E$, $\Phi$ is uniquely determined by those $(\vec E,u)\in
\cF$ such that $\# E_n$ is strictly increasing.  In Theorem~\ref{T.1} we have
seen that the Continuum Hypothesis implies the existence of an outer
automorphism determined by a coherent family of unitaries.
The following result, which is the main result of this section, complements Theorem~\ref{T.1}.

\begin{theorem}\lbl{P.1}
Assume TA. Then every automorphism of $\cC(H)$ determined by a
coherent family of unitaries is inner.
\end{theorem}

We shall first show that it suffices to prove Theorem~\ref{P.1} 
in the case when each $u$ is of the form $u_\alpha$ for $\alpha\in \cUN$, 
as constructed in \S\ref{S.CH}. 
If $\Phi$ is determined by $\cF$,  fix $(\vec E_0,u_0)\in \cF$. Then
$\cF'=\{(\vec F,v(u_0)^*)\mid  (\vec F,v)\in \cF\}$ is a coherent family of
unitaries. The automorphism $\Phi'$ determined by $\cF'$ is inner if and only
if $\Phi$ is inner. Also, $\Phi'$ is equal to the identity on the standard
atomic masa.  In the proof of
Theorem~\ref{P.1} we may therefore assume $\Phi$ is equal to the identity
on the standard atomic masa. Recall that ${\bbT}$ is the circle group.
For $\alpha\in \cUN$ by $u_\alpha$ denote the
unitary that sends $e_n$ to $\alpha_ne_n$. By our convention and
Lemma~\ref{L.A.3}, for every $(\vec E,u)\in \cF$ there is $\alpha$ such that
$\Psi_{u_\alpha}$ and $\Psi_u$ agree modulo compacts on $\cB(H)$. We may
therefore identify $\cF$ with the family $\{(\vec E,\alpha)\mid (\vec E,u)\in
\cF, \Psi_u$ and $\Psi_{u_\alpha}$ agree modulo compacts$\}$. It will also be
convenient to code partitions $\vec E$ by functions $f\colon \bbN\to \bbN$.

\subsection{The directed set $(\NupN,\leq^*)$} Let $\NupN$ denote the set of
all strictly increasing functions $f\colon \bbN\to \bbN$ such that $f(0)>0$.
(The reader should be warned that the  requirement that $f(0)>0$ is \emph{nonstandard} and 
\emph{important}.)
Such a function can code a partition of $\bbN$ into finite intervals in more
than one way.
 It will be convenient to use the following quantifiers:
$(\forall^\infty n)$ stands for $(\exists n_0)(\forall n\geq n_0)$ and
$(\exists^\infty n)$ stands for the dual quantifier, $(\forall n_0)(\exists
n\geq n_0)$. For $f$ and $g$ in $\NupN$ write $f\geq^* g$ if $(\forall^\infty
n)f(n)\geq g(n)$. A diagonal argument shows that $\NupN$ is
\emph{$\sigma$-directed} in the sense that for each sequence  $(f_n)$ in
$\NupN$ there is $g\in \NupN$ such that $f_n\leq^* g$ for all $n$.

For $f\in \NupN$ recursively define $f^+$ by $f^+(0)=f(0)$ and
$f^+(i+1)=f(f^+(i))$. Some $\cX\subseteq \NupN$ is \emph{$\leq^*$-cofinal} if
$(\forall f\in \NupN)(\exists g\in \cX)f\leq^* g$.

\begin{lemma}\lbl{L.NupN} Assume $\cX\subseteq \NupN$ is $\leq^*$-cofinal.
\begin{enumerate}
\item If  $\cX$ is partitioned into countably many pieces,
then at least one of the pieces is $\leq^*$-cofinal.
\item $(\exists^\infty n)(\exists i)(\forall k\geq n)(\exists f\in\cX)
(f(i)\leq n$ and $f(i+1)\geq k)$.
\item $\{f^+\mid f\in \cX\}$ is $\leq^*$-cofinal.
\end{enumerate}
\end{lemma}

\begin{proof} (1) Assume the contrary, let $\cX=\bigcup_n \cY_n$
be such that no $\cY_n$ is cofinal. Pick $f_n$ such that $f_n\not\leq^* g$ for
all $g\in \cY_n$. If $f\geq^* f_n$ for all $n$, then there is no $g\in \cX$
such that $f\leq^* g$---a contradiction.

(2) We first prove
\begin{enumerate}
\item [(*)]$(\exists^\infty n)(\forall k\geq n)(\exists i)(\exists f\in \cX)(f(i)\leq n$
and $f(i+1)\geq k)$.
\end{enumerate}
Assume not and fix $n_0$ such that for all $n\geq n_0$ there is $k=g(n)$ such
that for all $f\in \cX$ and all $i$, if $f(i)\leq n$ then $f(i+1)\leq g(n)$.
Define functions $h_m$  for $m\in \bbN$ recursively by $h_m(0)=\max(m,n_0)$ and
$h_m(i+1)=g(h_m(i))$. For $f\in \cX$ we have $f\leq^* h_{f(0)}$. By recursion
we prove $f(i)\leq h_{f(0)}(i)$ for all $i$. For $i=0$ this is immediate.
Assume $f(i)\leq h_{f(0)}(i)$. Then $f(i+1)\leq g(f(i))\leq
g(h_{f(0)}(i))=h_{f(0)}(i+1)$. Now fix $h\in \NupN$ such that $h_m\leq^* h$. By
the above, there is no $f\in \cX$ such that $h\leq^* f$, a contradiction.

For each $n$ the set $\{i\mid (\exists f\in \cX)f(i)\leq n\}$ is finite.
Therefore in (*) the same $i$ works for infinitely many $k$. An easy induction
shows that for  $f\in \NupN$ we have
 $f(i)\leq f^+(i)$
for all $i$, and (3) follows.
\end{proof}

\begin{lemma} \lbl{L.A.-1} If $f,g\in \NupN$ and $f\geq^* g$ then for all but finitely
many $n$ there is $i$ such that $f^i(0)\leq g(n)<g(n+1)\leq f^{i+2}(0)$. If
moreover $f(m)\geq g(m)$  for all $m$, then for every $n$ there is such an $i$.
\end{lemma}

\begin{proof} If $n$ is such that $f(m)\geq g(m)$ for all $m\geq n$,
let $i$ be the minimal such that $f^{i+1}(0)\geq g(n)$. Then $f^{i+2}(0)\geq
f(g(n))$, and since $g\in \NupN$ implies $g(n)\geq n+1$ this is $\geq
f(n+1)\geq g(n+1)$.
\end{proof}

 To $f\in \NupN$ associate
sequences of finite intervals of $\bbN$:
\begin{align*}
E^f_n&=[f(n),f(n+1))\\
F^f_n&=[f^n(0),f^{n+2}(0))\\
F^{f,\even}_n&=[f^{2n}(0),f^{2n+2}(0))\\
F^{f,\odd}_n&=[f^{2n+1}(0), f^{2n+3}(0))
\end{align*}
(`$F$' is for `fast'). By Lemma~\ref{L.A.-1}, if $\cX\subseteq \NupN$ is
$\leq^*$-cofinal, then each one of $\{\vec F^{f,\even}\mid f\in \cX\}$ and
$\{\vec F^{f,\odd}\mid f\in \cX\}$ is a cofinal family of partitions as defined
in \S\ref{S.CH}. Notation $\Delta_I(\alpha,\beta)$ used in the following proof
was defined before Lemma~\ref{L.A.5.5}.

\begin{lemma}\lbl{L.A.1/2} Assume $\Phi$ is an automorphism of $\cC(H)$
whose restriction to the standard atomic masa is equal to the identity and
which is determined by a coherent family of unitaries. For each $f\in \NupN$
there is $\alpha\in \cUN$ such that $\Psi_\alpha$ is a representation of $\Phi$
on both $\calD[\vec F^{f,\even}]$ and $\calD[\vec F^{f,\odd}]$.
\end{lemma}

\begin{proof}
 By Lemma~\ref{L.A.3} for each $f$ there are
$\beta$ and $\gamma$ in $\cUN$ such that $\Psi_\beta$ is a representation of
$\Phi$ on $\calD[\vec F^{f,\even}]$ and $\Psi_\gamma$ is a representation of
$\Phi$ on $\calD[\vec F^{f,\odd}]$. Define $\beta'$ and $\gamma'$ recursively
as follows. For $i\in [f^0(0), f^2(0))$ let $\beta'(i)=\beta(i)$. If
$\beta'(i)$ has been defined for $i<f^{2n}(0)$, then for $i\in
[f^{2n-1}(0),f^{2n+1}(0))$ let
$$
\gamma'(i)=\gamma(i)\overline{\gamma(f^{2n-1}(0))}\beta'(f^{2n-1}(0)).
$$
If $\gamma'(i)$ has been defined for  $i<f^{2n+1}(0)$
then for $i\in [f^{2n}(0),f^{2n+2}(0))$ let
$$
\beta'(i)=\beta(i)\overline{\beta(f^{2n}(0))}\gamma'(f^{2n}(0)).
$$
Then $\gamma'(f^j(0))=\beta'(f^j(0))$ for all $j$ and
$\Psi_{\beta'}(a)=\Psi_\beta(a)$ for all $a\in \calD[\vec F^{f,\even}]$ and
$\Psi_{\gamma'}(a)=\Psi_\gamma(a)$ for all $a\in \calD[\vec F^{f,\odd}]$. Let
$J_n=[f^n(0),f^{n+1}(0))$. 
Then $\sup_{i\in
J_n}|\beta'(i)-\gamma'(i)|\leq \Delta_{J_{n}}(\beta,\gamma')$ by
Lemma~\ref{L.A.5.5}(2). Since $\Psi_{\beta'}$, $\Psi_\beta$, $\Psi_{\gamma'}$
and $\Psi_\gamma$ are all representations of $\Phi$ on $\calD[\vec J]$, by
Lemma~\ref{L.A.4} (b) we conclude that $\Delta_{J_{2n+1}}(\beta',\gamma)\to 0$
and $\Delta_{J_{2n}}(\beta,\gamma')\to 0$ as $n\to \infty$. Therefore  $\lim_i
|\beta'(i)-\gamma'(i)|=0$, and therefore Lemma~\ref{L.A.4} (a) implies that
$\Psi_{\gamma'}$ and $\Psi_{\beta'}$ agree on $\cB(H)$ modulo the compact
operators. Therefore $\alpha=\beta'$ is as required.
\end{proof}

\begin{proof}[Proof of Theorem~\ref{P.1}]
As pointed out after the statement of Theorem~\ref{P.1}, we may assume $\Phi$ is equal to
the identity on the standard atomic masa 
and that the unitary $u$ is of the form $u_\alpha$
for each $(f,u)$ in the
coherent family defining $\Phi$.  
Let $\cX\subseteq \bbN^{\uparrow
\bbN}\times {\bbT}^{\bbN}$ be the set of all pairs $(f,\alpha)$ such that
$\Psi_\alpha$ is a representation of $\Phi$ on both $\calD[\vec F^{f,\even}]$
and $\calD[\vec F^{f,\odd}]$. By Lemma~\ref{L.A.1/2} for every $f\in \NupN$
there is~$\alpha$ such that $(f,\alpha)\in \cX$. 

For $\e>0$ define $[\cX]^2=K^\e_0\cup K^\e_1$
by assigning a pair $(f,\alpha)$, $(g,\beta)$ to $K^\e_0$
if 
\begin{enumerate}
\item [($K^\e_0$)] there are $m,n$ such that with
$J=F^f_m\cap F^g_n$ we have
$\Delta_J(\alpha,\beta)>\e$.
\end{enumerate}
 We consider $\bbN^{\uparrow \bbN}$ with the \emph{Baire space topology},
 induced by the metric
 $$
 d(f,g)=2^{-\min\{n\mid f(n)\neq g(n)\}}.
 $$
This is a complete separable metric. Consider $\cUN$ in the product of strong
operator topology and $\cX$ in the product  topology.  If
$K_0^\e$ is identified with a symmetric subset of $\cX^2$ off the diagonal,
then it is open in this topology.

\begin{claim}\lbl{C.A.1}  TA implies that
for  $\e>0$ there are no  uncountable $K^\e_0$-ho\-mo\-ge\-ne\-ous
subsets
 of $\cX$.
\end{claim}

\begin{proof} Assume the contrary and fix $\e>0$ and
an uncountable $K^\e_0$-ho\-mo\-ge\-ne\-ous set $\calH$. We shall
refine $\calH$ to an uncountable subset several times, until we reach a
contradiction. In order to keep the notation under control, each
successive refinement will be called $\calH$. 
Let
\[
\cF =\{g^{+}\mid (\exists \alpha)(g,\alpha)\in \calH\}.
\]
We may assume $\calH$
has size~$\aleph_1$, hence TA and \cite[Theorem~3.4 and
Theorem~8.5]{To:Partition} imply that  $\cF$ is
 $\leq^*$-bounded by some $\bar f\in \NupN$.
For each $g\in \cF$ fix~$l_g$ such that $\bar f(n)\geq g(n)$ for all $n\geq l_g$ and
let $s_g=g\rs l_g$. Fix $(\bar l,\bar s)$ such that $\{g\in \cF\mid
(l_g,s_g)=(\bar l,\bar s)\}$ is uncountable. By refining $\calH$ and 
increasing $\bar f\rs \bar l$ to
$\bar f(\bar l)$ we may assume $\bar f(n)\geq g^+(n)$ for all $g^+\in \cF$ and all
$n\in \bbN$. 
 Lemma~\ref{L.A.-1} implies that for every $(g,\alpha)\in \calH$ and every
$n$ there is $i$ such that $F^g_n\cup F^g_{n+1}\subseteq F^{\bar f}_i\cup F^{\bar f}_{i+1}$. 
By
Lemma~\ref{L.A.1/2} we may fix $\alpha\in \cUN$ such that~$\Psi_\alpha$ is a
representation of $\Phi$ on both $\calD[\vec F^{\bar f,\even}]$ and $\calD[\vec
F^{\bar f,\odd}]$. Lemma~\ref{L.A.4} (b) implies that  for every $(g,\beta)\in
\calH$ we have $\limsup_n \Delta_{E^g_n}(\alpha,\beta)=0$. Fix $\bar k, \bar
m\in \bbN$ for which the set $\calH'$ of all $(g,\beta)\in \calH$ such that
$g^{\bar m+1}(0)=\bar k$ and $\Delta_{E^g_n}(\alpha,\beta)<\e/2$ whenever $n\geq
\bar m$ is uncountable. By the separability of $\cUN$ we can find distinct
$(g,\beta)$ and $(h,\gamma)$ in $\calH'$ such that $g^i(0)=h^i(0)$
for all $i\leq \bar m+1$ and $|\beta(i)-\gamma(i)|<\e/2$  for all $i\leq \bar k$.

Fix $i$ and $j$ such that $J=F^g_i\cap F^h_j\neq \emptyset$. 
We shall prove  $\Delta_J(\beta,\gamma)<\e$. 
Since  $\max_{i<\bar k}|\beta(i)-\gamma(i)|<\e/2$, we may assume that $J\setminus [0,\bar k)\neq \emptyset$ and therefore $J\cap [0,g^{\bar m}(0))=J\cap [0,h^{\bar m}(0))=\emptyset$. 
Find $l$ such that $F^g_i\subseteq F^{\bar f}_l$, and therefore $J\subseteq F^{\bar f}_l$. 
Then 
\[
\Delta_J(\beta,\gamma)\leq \Delta_{F^g_i\cap F^{\bar f}_l}(\beta,\alpha)+
\Delta_{F^h_j\cap F^{\bar f}_l}(\gamma,\alpha)<\e.
\]
Since $i$ and $j$ were arbitrary we conclude that 
  $\{(g,\beta),(h,\gamma)\}\in K^\e_1$ contradicting our assumption on $\calH$. 
\end{proof}

Since $K^\e_0$ is an open partition,  by TA and Claim~\ref{C.A.1},
for each $\e>0$ there is a partition of $\cX$ into countably many
$K^\e_1$-homogeneous sets.

We fix $n$ and 
 let $\e_n=2^{-n}$. Repeatedly using Lemma~\ref{L.NupN},
 find sets
$\cX_0\supseteq \cX_1\supseteq\dots$ and $0=m(0)<m(1)<\dots$ so that (1) each
$\cX_n$ is $K^{\e_n}_1$-homogeneous, (2) each set $\{f\mid (\exists
\alpha)(f,\alpha)\in \cX_n\}$ is $\leq^*$-cofinal  and (3) for all $n$ and
 all $k>m(n)$ there are $j\in\bbN$ and $(f,\alpha)\in
\cX_n$ such that $f^j(0)\leq m(n)$ and $f^{j+1}(0)\geq k$.

In $\cX_n$ pick a sequence $(f_{n,i},\alpha_{n,i})$ and $j(i)$, for $i\in \bbN$, such that
\begin{enumerate}
\setcounter{enumi}3
\item  \lbl{I.new.0}
$(f_{n,i})^{j(i)}(0)\leq m(n)<m(n+i)\leq (f_{n,i})^{j(i)+1}(0)$ for all $i$.  
\pushcounter
\end{enumerate}
By compactness
we may assume  $\alpha_{n,i}$ converge to  $\alpha_n\in \cUN$. 
We claim that 
\begin{enumerate}
\popcounter
\item \lbl{I.new.1} $\Delta_{[m(l),\infty)}(\alpha_k,\alpha_l)\leq \e_k$ whenever $k\leq l$. 
\pushcounter
\end{enumerate}
Assume not and fix $m(l)\leq n_1<n_2$ such that
$\rho(n_1,n_2,\alpha_k,\alpha_l)>\e_k$. By \eqref{I.new.0}, for all large enough $i$ we have
$(f_{k,i})^{j(i)+1}(0)>n_2$ and 
$(f_{l,i})^{j(i)+1}(0)>n_2$. Since $\lim_i\alpha_{k,i}= \alpha_k$ and $\lim_i \alpha_{l,i}=\alpha_l$ 
we have $\rho(n_1,n_2,\alpha_{k,i},\alpha_{l,i})>\e_k$ for large enough $i$. 
These facts imply $\Delta_{F^{f_{k,i}}_{j(i)}\cap F^{f_{l,i}}_{j(i)}}(\alpha_{k,i},\alpha_{l,i})>\e_k$ 
for a large enough~$i$. However, $(f_{k,i},\alpha_{k,i})\in \cX_k$ and $(f_{l,i},\alpha_{l,i})\in\cX_l
\subseteq \cX_k$, contradicting the homogeneity of~$\cX_k$. 

By \eqref{I.new.1} and Lemma~\ref{L.A.5.5} \eqref{L.A.5.5.new},  
for $k<l$ we can fix $z_{k,l}\in\bbT$ such that 
\begin{enumerate}
\popcounter
\item   \lbl{I.new.2} $\sup_{i\geq m(l)} |z_{k,l} \alpha_k(i) -\alpha_l(i)|\leq \e_k$, 
\pushcounter
\end{enumerate}
with $z_{k,k}=1$. 
We claim that 
\begin{enumerate}
\popcounter
\item   \lbl{I.new.3} $|z_{k,l}-z_{k,m}\overline{ z_{l,m}}|\leq 3\e_{\min\{k,l,m\}}$ for all $k,l$ and $m$. 
\pushcounter
\end{enumerate}
For $\beta$ and $\gamma$ in $\cUN$ and $\e>0$, 
in the proof of \eqref{I.new.3} we write $\beta\sim_\e \gamma$ if 
$\sup_{i\geq m(\max\{k,l,m\})}|\beta(i)-\gamma(i)|\leq \e$. 
Letting $\e=\e_{\min\{k,l,m\}}$, 
 by \eqref{I.new.2}  we have
\[
z_{k,l}\alpha_k\sim_\e\alpha_l \sim_\e\overline{ z_{l,m}} \alpha_m\sim_\e \overline{ z_{l,m}}z_{k,m}\alpha_k
\]
and therefore $|z_{k,l}-z_{k,m}\overline{ z_{l,m}}|\leq 3\e$. 

We want to find an infinite $Y\subseteq \bbN$ such that 
for all $i<j$ in $Y$ we have $|1-z_{k(i),k(j)}|\leq 4\e_{k(i)}$. 
To this end, define a coloring  $M_0\cup M_1$ of triples in~$\bbN$  by 
putting a triple $k<l<m$ into $M_0$ if 
\[
|z_{l,m}-1|\leq 4\e_k.
\]
We claim there are no  infinite sets $Y$ such that every triple of elements
from~$Y$ belongs to $M_1$. Assume the contrary. For such $Y$ let $k=\min(Y)$ 
and pick $n\in Y$ such that $Y$ has at least $2\pi/\e_k$ strictly between $k$ and $n$.
Then there are distinct $l$ and $m$ in $Y$ between $k$ and $n$ such that 
$|z_{l,n}-z_{m,n}|\leq \e_k$. Using  \eqref{I.new.3} in the second inequality 
we have 
\[
|z_{l,m}-1|\leq |z_{l,m}-z_{l,n}\overline{ z_{m,n}}|+|1-z_{l,n}\overline{ z_{m,n}}|\leq 4\e_k.
\]
Therefore  there is no infinite $M_1$-homogeneous set of triples.  By
using Ramsey's theorem we can find an infinite $Y\subseteq \bbN$
such that all triples of elements of $Y$ belong to $M_0$. 
Enumerate $Y$ increasingly as $k(i)$, for $i\in \bbN$.
We may assume $k(0)\geq 2$ and therefore $4\e_{k(i)}\leq \e_i$. 
Since $|1-z_{k(i),k(j)}|\leq \e_i$, we have 
\begin{enumerate}
\popcounter
\item \lbl{I.new.10}
$\sup_{l\geq m(k(i))}|\alpha_{k(i)}-\alpha_{k(j)}|\leq\e_i'$  for all $i<j$. 
\pushcounter
\end{enumerate}
Define $\gamma\in \cUN$ by 
$\gamma(l)=\alpha_{k(i)}(l)$ if 
$l\in [m(k(i)),m(k(i+1)))$ and $\gamma(l)=\alpha_{k(0)}(l)$ if $l<m(k(0))$. 
By \eqref{I.new.10}  we have 
\begin{enumerate}
\popcounter
\item \lbl{I.new.4} $|\gamma(l)-\alpha_{k(i)}(l)|\leq \e_i$ for all $i$ and all $l\geq m(k(i))$. 
\pushcounter
\end{enumerate}
We claim that for all $j$   (recall that $F^f_i=[f^i(0),f^{i+1}(0))$) 
\begin{enumerate}
\popcounter
\item \lbl{I.A.7} If $(f,\beta)\in\cX_{k(j)}$ then for all $i$
we have $\Delta_{F^f_i\setminus m_{k(j)}}(\beta,\gamma)\leq
5\e_j$. \pushcounter
\end{enumerate}
Write $n=k(j)$. 
Since $(f_{n,l},\alpha_{n,l})\in \cX_n$, 
 for  $l$  large enough to have 
\[
[(f_{n,l})^{j(l)}(0), (f_{n,l})^{j(l)+1}(0))\supseteq F^f_i\setminus m(n)
\]
 we have 
$\Delta_{F^f_i\setminus m(n)}(\beta,\alpha_{n,l})\leq \e_n$.
 The continuity implies  $\Delta_{F^f_i\setminus m(n)}(\beta,\alpha_n)\leq \e_n$ and    
 \eqref{I.A.7}  implies 
 $\Delta_{F^f_i\setminus m_{k(j)}}(\beta,\gamma)\leq
5\e_j$.

By Lemma~\ref{L.U.1}, it suffices to prove  $\Psi_\gamma$ is a representation
of $\Phi$ on every~$\calD[\vec E]$. Fix $g$ such that $\vec E=\vec E^g$ (with
$E^g_n=[g(n),g(n+1))$). Fix $\beta$ such that $\Psi_\beta$ is a representation
of $\Phi$ on $\calD[\vec E^g]$.  By Lemma~\ref{L.A.4} (b) it suffices to prove
$\Delta_{E^g_n}(\beta,\gamma)\to 0$ as $n\to \infty$. Fix $m\in \bbN$ and
$(f,\alpha)\in \cX_m$ such that $f\geq^* g$. By Lemma~\ref{L.A.4} (b) we have
$\lim_n \Delta_{E^g_n}(\alpha,\beta)\to 0$. By~\eqref{I.A.7} we have $\limsup_n
\Delta_{E^g_n}(\beta,\gamma)\leq \limsup_n \Delta_{F^f_n}(\alpha,\gamma)\leq
5\e_n$. Since $n$ was arbitrary, the conclusion follows.
\end{proof}

The construction in Theorem~\ref{T.1} hinges on the existence of a
nontrivial coherent family of unitaries under CH and
Theorem~\ref{P.1} shows that under TA every coherent family of
unitaries is `uniformized' by a single unitary. There is an analogy
to the effect of CH/TA  on the additivity of the strong homology as
exhibited in \cite{MarPra}/\cite{DoSiVa}  and
\cite[Theorem~8.7]{To:Partition}. In the latter, uniformizing certain
families of functions from subsets of $\bbN$ into $\{0,1\}$ that are
coherent modulo finite plays the key role. For more on such
uniformizations  see \cite[\S\S 2.2--2.4]{Fa:AQ}.

\section{Representations and $\e$-approximations}\lbl{S.Representation}

The present section is a loose collection of results showing that a
sufficiently measurable representation, or an approximation to a
representation, can be further improved in one way or another.

Lemma~\ref{L.trivial} below illustrates how drastically different automorphisms
of the Calkin algebra are from the automorphisms of Boolean algebras
$\cP(\bbN)/{\mathcal I}$. It is directly responsible for the fact that the
Martin's Axiom is not needed in the proof of Theorem~\ref{T.0}. Recall that for
$\calD\subseteq \cB(H)$ we say $\Phi$ is \emph{inner on~$\calD$} if there is an
inner automorphism $\Phi'$ of $\cC(H)$ such that the restrictions of $\Phi$ and
$\Phi'$ to $\pi[\calD]$ coincide.

\begin{lemma}\lbl{L.trivial}  Assume $\calD_1$ and $\calD_2$ are subsets of $\cB(H)$
such that for some partial isometry $u$ we have $u\calD_2u^*\subseteq \calD_1$
and $P=u^*u$ satisfies $P\calD_2 P=\calD_2$.  If $\Phi$ is inner on $\calD_1$,
then it is inner on $\calD_2$.
\end{lemma}

\begin{proof}
Fix $v$ such that $a\mapsto vav^*$ is a representation of $\Phi$ on $\calD_1$
and $w$ such that $\pi(w)=\Phi(\pi(u))$. If $b\in \calD_2$ then $ubu^*\in
\calD_1$ and $u^*ubu^*u=b$. If $\Psi$ is any representation of $\Phi$, then we
have (writing $c\sim_{\cK}d$ for `$c-d$ is compact')
$$
\Psi(b)\sim_{\cK} \Psi(u^*)\Psi(ubu^*)\Psi(u)\sim_{\cK}w^* vub u^* v^* w.
$$
Therefore $w^*vu$ witnesses $\Phi$ is inner on $\calD_2$.
\end{proof}

An analogue of Lemma~\ref{L.trivial} fails for automorphisms of
$\cP(\bbN)/\Fin$. For example,  in \cite{ShSte:Nontrivial} (see also
\cite{Ste:Size}) it was proved that a weakening of the Continuum Hypothesis
implies the existence of a nontrivial automorphism whose ideal of trivialities
is a maximal ideal.

\subsection{$\e$-approximations}\lbl{S.epsilon}
Assume $\cA$ and $\cB$ are C*-algebras,  $J_1$ and $J_2$ are their ideals,
$\Phi\colon \cA/J_1\to \cB/J_2$ is a
*-homomorphism, and $\cX\subseteq \cA$.
 A map $\Theta$ whose domain contains $\cX$ and is contained in $\cA$
 and whose range is contained in $\cB$ is an
 \emph{$\e$-approximation to $\Phi$ on $\cX$} if for all $a\in \cX$
 we have $\|\Phi(\pi_{J_1}(a))-\pi_{J_2}(\Theta(a))\|\leq
 \e$ for all $a\in \cX$.
 If $\cX=\cA$ we say $\Theta$ is an \emph{$\e$-approximation to
 $\Phi$}.

\begin{lemma} \lbl{L.L.1} Assume $\cA$ and $\cB$ are C*-subalgebras
of $\cB(H)$ containing $\cK(H)$ and $\Phi$ is a
*-homomorphism from $\cA/\cK(H)$ into $\cB/\cK(H)$. Then we
have the following.
\begin{enumerate}
\item $\Phi$ has a Borel-measurable representation if and only if
it has a Borel-measurable representation on $\cU(\cA)$.
\item If $\Phi$ has a Borel-measurable $\e$-approximation
on $\cU(\cA)$ then it has a Borel-measurable
$4\e$-approximation on $\cA_{\leq 1}$.
\end{enumerate}
\end{lemma}

\begin{proof} (1) We only need to prove the reverse implication.
There  are norm-continuous maps $\gamma_i\colon \cA_{\leq 1}\to \cU(\cA)$ for
$i<4$ such that $a=\sum_{i<4} \gamma_i(a)$ for all $a\in \cA$. This is because
if $a\in \cA$ then $b=(a+a^*)/2$ and $c=i(a-a^*)/2$ are self-adjoint of norm
$\leq \|a\|$ such that $b+ic=a$. If $b$ is self-adjoint of norm $\leq 1$, then
the operators $b_1=b+i\sqrt {I-b^2}$ and $b_2=b-i\sqrt {I-b^2}$ have norm $\leq
1$ and their product is equal to $I$. Therefore they are both unitaries. Also,
their mean is equal to $b$. It is now clear how to define~$\gamma_i$. Assume
$\Psi$ is a representation of $\Phi$ on $\cU(\cA)$. Then let $\Psi_1(0)=0$ and
$\Psi_1(a)=\|a\|\sum_{i<4} \Psi(\gamma_i(a/\|a\|))$ for $a\neq 0$. This is the
required Borel representation. The proof of (2) uses the same formula and the
obvious estimates.
\end{proof}

\begin{lemma} \lbl{L.guess} Let $\calD$ be
a von Neumann subalgebra of $\cB(H)$ and $\Phi\colon
\calD/(\cK(H)\cap \calD)\to \cB(H)/\cK(H)$ a *-homomorphism.

\begin{enumerate}
\item \lbl{L.guess.1} If $\Phi$ has a
 C-measurable $\e$-approximation $\Psi$ on $\cU(\calD)$ then it has a
Borel-measurable  $8\e$-ap\-pro\-xi\-ma\-ti\-on on
$\cU(\calD)$.

\item \lbl{L.guess.1.5}If $\Phi$ has a
 C-measurable representation  on $\cU(\calD)$ then it has a
Borel-mea\-su\-ra\-ble  representation on $\calD$.

\item \lbl{L.guess.2} If there are C-measurable maps $\Psi_i$ for $i\in \bbN$
whose graphs cover an $\e$-approximation to~$\Phi$ on $\calD_{\leq
1}$ then there are Borel-measurable maps $\Psi'_i$ for $i\in \bbN$
whose graphs cover an $8\e$-approximation to~$\Phi$ on $\calD_{\leq
1}$.
\end{enumerate}
\end{lemma}

\begin{proof} \eqref{L.guess.1} Consider $\cU(\calD)$ with respect to the strong operator
topology. It is  a Polish group. Since $\Psi$ is Baire-measurable we may fix a
dense $G_\delta$ subset $\cX$ of $\cU(\calD)$ on which $\Psi$ is continuous.
The set 
\[
\cY=\{(a,b)\in \cU(\calD)^2\mid b\in \cX\cap a^*\cX\}\}
\]
 is Borel and it
has comeager sections. By Theorem~\ref{T.U.Large} there is a Borel
uniformization $h$ for $\cY$. Then for each $a$ both $h(a)$ and $ah(a)$ belong to $\cX$ and therefore
$\Psi'(a)=\Psi(ah(a))\Psi(h(a))^*$ is a
Borel-measurable $2\e$-approximation to $\Phi$ on $\cU(\calD)$. By
Lemma~\ref{L.L.1}, $\Phi$ has an $8\e$-approximation.

\eqref{L.guess.1.5} follows from the case $\e=0$ of \eqref{L.guess.1} plus
Lemma~\ref{L.L.1}(1).

To prove \eqref{L.guess.2}, find a dense $G_\delta$ subset $\cX$ of
$\cU(\calD)$ on which each $\Psi_i$ is continuous. Define $\cY$ and $h$ as
above and consider the maps $\Psi_{ij}'(a)=\Psi_i(ah(a))\Psi_j(h(a))^*$.
\end{proof}

\begin{lemma} \lbl{L.n}
Let $\calD\subseteq \cB(H)$ be  a von Neumann algebra, $\Phi\colon
\calD/(\cK(H)\cap \calD)\to \cB(H)/\cK(H)$ a
*-homomorphism and $\cY\subseteq \calD_{\leq 1}$.  Assume $\Phi$ has a Borel-measurable
$2^{-n}$-approximation $\Xi_n$ on $\cY$ for every $n\in \bbN$. Then $\Phi$ has
a C-measurable representation on $\cY$.
\end{lemma}

\begin{proof} Let $\cX=\{(a,b)\in
\calD_{\leq 1}\times \cB(H)_{\leq 1}\mid (\forall n) \|\Xi_n(a)-b\|_{\cK}\leq
2^{-n+1}\}$. By Lemma~\ref{L.T.3}, this is a Borel set. If $\Psi$ is a
C-measurable uniformization of $\cX$ provided by Theorem~\ref{T.JvN}, then
$\Psi$ is a representation of $\Phi$ on $\cY$.
\end{proof}

\section{Approximate *-homomorphisms}
\lbl{S.approximate}

 Assume $\cA$ and $\cB$ are C*-algebras.
 A map $\Lambda\colon \cA\to \cB$ in an \emph{$\e$-approximate *-homomorphism}
 if for all $a,b$ in $\cA_{\leq 1}$ we have the following.
 \begin{enumerate}
 \item $\|\Lambda(ab)-\Lambda(a)\Lambda(b)\|<\e$,
 \item $\|\Lambda(a+b)-\Lambda(a)-\Lambda(b)\|<\e$,
 \item $\|\Lambda(a^*)-\Lambda(a)^*\|<\e$,
 \item $|\|a\|-\|\Lambda(a)\||<\e$.
 \end{enumerate}
 We say $\Lambda$ is \emph{unital} if both $\cA$ and $\cB$ are unital and
 $\Lambda(I)=I$.
 We say $\Lambda$ is \emph{$\delta$-approximated} by $\Theta$ if
 $\|\Lambda(a)-\Theta(a)\|<\delta$ for all $a\in \cA_{\leq 1}$.
Theorem~\ref{T.approximate} is the main result of this section and may be of
independent interest. The numerical value of the constant $K$ is irrelevant for
our purposes and we shall make no attempt to provide
a sharp bound.

A shorter proof of Theorem~\ref{T.approximate} can be given by using a special
case of a result of Sakai (\cite{Sak:On}). After applying the
Grove--Karcher--Roh/Kazhdan result on $\e$-representations to obtain a
representation $\Theta$ of $\Lambda\rs \cU(\cA)$ that is a norm-continuous
group homomorphism, use \cite{Sak:On} to extend $\Theta$ to a *-homomorphism or
a conjugate
*-homomorphism of $\cA$ into $\cB$. An argument included in the proof below shows that this extension has to be a *-homomorphism.
Parts of the proof of Theorem~\ref{T.approximate} resemble parts of Sakai's
proof, of which I was not aware at the time of preparing this manuscript.

 \begin{thm}\lbl{T.approximate} There is a universal constant $K<\infty$
 such that the following holds.
 If $\e<1/1000$, $\cA$ is a  finite-dimensional C$^*$-algebra, $m\in \bbN$,  and
 $\Lambda\colon \cA\to M_m(\bbC)$ is a Borel-measurable unital $\e$-approximate
 *-homomorphism,
 then $\Lambda$ can be $K\e$-approximated by a unital *-ho\-mo\-mor\-phi\-sm.
 \end{thm}

In the terminology of S. Ulam, the approximate
*-ho\-mo\-mor\-phisms are \emph{stable} (see e.g., \cite{KanRe:Ulam}).
Connection between lifting theorems and Ulam-stability of approximate
homomorphisms between Boolean algebra was first exploited in
\cite{Fa:AQ}. Analogous results for groups appear in
\cite{Fa:Liftings} and \cite{KanRe:Ulam}. The following a special
case of a well-known result (see \cite{Lor:Lifting}) but we include a
proof for reader's convenience.

\begin{lemma} \lbl{L.approximate.unitary} Assume $\e<1$ and $a$ is an element of a finite-dimensional C*-algebra $A$  such that
$\|a^*a-I\|<\e$. If $a=bu$ is the polar decomposition of $a$ then $u$ is a
unitary and $\|a-u\|<\e$.
\end{lemma}

\begin{proof} We have $\|u^*b^2u-I\|<1$. Then $P=u^*u$ is a projection 
and $\|I-P\|=\|(I-P)(u^*bu-I)\|<1$. Therefore $u^*u=I$ and since $A$ is
finite-dimensional~$u$ is a unitary.
Hence we have $\|b^2-I\|<\e$ and  $\|b-I\|<\e$. Clearly,
$\|a-u\|=\|bu-u\|=\|b-I\|$.
\end{proof}

\newcommand{\eone}{2\e+\e^2}
 \begin{proof}[Proof of Theorem~\ref{T.approximate}]
  We shall write $\cB$ for $M_m(\bbC)$ and consider its representation on
  the $m$-dimensional Hilbert space, denoted $K$. 
We also write  $a\approx_\delta b$ for $\|a-b\|<\delta$.
 Fix a unitary $u$ in $\cA$ and let $a=\Lambda(u)$.
 Then $aa^*\approx_{\e(1+\e)}\Lambda(u)\Lambda(u^*)\approx_{\e} \Lambda(uu^*)$,
 thus $\|aa^*-I\|<\eone$.
 Similarly $\|a^*a-I\|<\eone$. Therefore by
 Lemma~\ref{L.approximate.unitary} there is a unitary $v\in \cB$ such that
 $\|\Lambda(u)-v\|<\eone=\e_1$.

 Let $\cX$  be the set of all pairs $(u,v)\in \cU(\cA)\times \cU(\cB)$
 such that $\|\Lambda(u)-v\|<\e_1$. Since $\Lambda$ is Borel-measurable, $\cX$ is a Borel set.
 By Theorem~\ref{T.JvN} there
 is a C-measurable $\Lambda'\colon \cU\to \cV$ uniformizing $\cX$.
 Note that $\|\Lambda'(u)-\Lambda(u)\|<\e_1$ for all unitaries $u$.

\newcommand{\etwo}{7\e+3\e^2}
We have $\Lambda'(u)\Lambda'(v)\approx_{(2+\e)\e_1}
\Lambda(u)\Lambda(v)\approx_\e\Lambda(uv)\approx_{\e_1} \Lambda'(uv)$. Thus
$\|\Lambda'(uv)-\Lambda'(u)\Lambda'(v)\|<(3+\e)\e_1+\e=\e_2$.  In the
terminology of \cite{Kaz:epsilon}, $\Lambda'$ is a
\emph{$2\e_2$-representation} of $\cU(\cA)$ on $K$. In \cite{Kaz:epsilon} and
\cite{GrKaRo:Jacobi} it was proved (among other things) that if $\delta<1/100$
then every strongly continuous $2\delta$-representation $\rho$ of a compact
group can be $2\delta$-approximated by a (strongly continuous) representation
$\rho'$. A more streamlined presentation of this proof is given in
\cite[Theorem~5.13]{AGG:Approximation}. The approximating representation is obtained as a
limit of a succession of integrals with respect to the Haar measure and the
assumption that  $\rho$ is continuous can be weakened to the assumption that
$\rho$ is Haar measurable without altering the proof (or the conclusion). In
particular, the proof given in \cite{AGG:Approximation} taken verbatim covers
the case of a Haar-measurable approximation.

\newcommand{\ethree}{15\e+7\e^2}
 Let $\Theta$ be a  continuous homomorphism between the
unitary groups of $\cA$ and $\cB$ that is a $2\e_2$-aproximation to $\Lambda'$
on $\cU(\cA)$. 
 For all $u$
we then have 
\[
\|\Theta(u)-\Lambda(u)\|<2\e_2+\e_1=\e_3.
\]
For a self-adjoint $a\in \cA$ the map $\bbR\ni r\mapsto \exp(ir a)\in \cU(K)$
defines  a norm-continuous one-parameter group. By Stone's theorem 
(e.g.,  \cite[Theorem 5.3.15]{Pede:Analysis}),  there is the unique $\rho(a)\in \cU(K)$
 such that 
\[
\Theta(\exp(ir a))=\exp(ir \rho(a)).
\]
Since  $\cB$ is a von Neumann algebra, we can conclude that $\rho(a)\in \cB$. 

\begin{claim} \label{C.eir} $\rho(I)=I$, hence for all $r\in \bbR$ we have $\Theta(e^{ir} I)=e^{ir} I$. 
\end{claim} 

\begin{proof} Let $b=\rho(I)$ and assume $b\neq I$. 
Since $b$ is self-adjoint, there is $s\neq 1$ in the spectrum of $b$. 
Pick $r\in \bbR$ such that $r(1-s)=\pi+2 k \pi$ for some $k\in \bbN$. 
Let $\xi$ be  the unital eigenvector of $\exp(irb)$ corresponding to the
eigenvalue $e^{irs}$. 
Then the vector 
\[
\exp(irb)(\xi)-e^{ir}(\xi)=e^{irs} \xi-e^{ir}\xi=e^{irs}(\xi-e^{ir(1-s)}\xi)=2e^{irs}\xi
\]
 has norm 2, 
hence 
 $\|\exp(irb)-e^{ir} I\|=2$. 
However, 
\[
\|\Theta(e^{ir} I)-e^{ir} I\|\leq \|\Theta(e^{ir} I)-\Lambda(e^{ir} I)\|+\|\Lambda(e^{ir} I)-e^{ir} I\|
\leq \e_3+\e<1
\]
a contradiction. 
\end{proof}

Let $u$ be a self-adjoint unitary. Then $u=I-2P$ for some projection~$P$
and therefore $\exp(ir u)=e^{ir} u$ and by  Claim~\ref{C.eir} we have $\Theta(u)=\rho(u)$. 
Also $P=\frac 12 (I-u)$ and one straightforwardly checks
that $\rho(P)=\frac 12(I-\rho(u))$ and 
that $\|\rho(P)-\Lambda(P)\|\leq \e$.

 \begin{claim} \label{C.PQ} 
 If  projections $P$ and $Q$ commute then
$\rho(P)$ and $\rho(Q)$ commute.
If  $PQ=0$ then $\rho(P)\rho(Q)=0$. 
\end{claim} 

\begin{proof}  
  Since $I-2P$ and $I-2Q$ commute if and
only if $P$ and $Q$ commute, the first sentence follows. 
If $PQ=0$, then 
$I-2\rho(P+Q)=\Theta((I-2P)(I-2Q))=(I-2\rho(P))(I-2\rho(Q))$,
we have
$\rho(P+Q)-\rho(P)-\rho(Q)=2\rho(P)\rho(Q)$.
The left hand side has norm $<4\e<1$. As a product of two commuting
projections, $\rho(P)\rho(Q)$ is a projection, so it has to be 0.
\end{proof} 

We say that projections $P$ and $Q$ in a C*-algebra $\cA$ 
are \emph{Murray--von Neumann equivalent} and  write $P\sim Q$ 
if there is a partial isometry $u$ such that $uPu^*=Q$. 
In the case when $\cA$ is finite-dimensional this is equivalent 
to asserting that there is a unitary $u$ such that $uPu^*=Q$.

\begin{claim} \label{C.MvN} 
If $P$ and $Q$ are Murray--von Neumann equivalent projections 
then $\rho(P)$ and $\rho(Q)$ are Murray--von Neumann equivalent 
projections. 
\end{claim} 

\begin{proof} Fix a unitary $u$ such that $uPu^*=Q$. Write $v_P=I-2P$ ad
$v_Q=I-2Q$. Then (with $w=\Theta(u)$) we have
 $w\Theta(v_P)w^*= \Theta(v_Q)$, 
and therefore $w\Theta(P)w^*=\Theta(Q)$. 
\end{proof}

Let $u$ be a unitary in $\cA$ and let $e^{ir_j}$, for $0\leq j\leq n-1$, 
be the spectrum of $u$. Fix projections $P_j$, for $0\leq j\leq n-1$, such that
$u=\sum_{j=0}^{n-1} e^{ir_j} P_j=\prod_{j=0}^{n-1} \exp(ir_j P_j)$. Then (using Claim~\ref{C.PQ} in the 
last equality)
\[
\textstyle\Theta(u)=\prod_{j=0}^{n-1} \Theta(\exp (ir_j P_j))=\prod_{j=0}^{n-1} \exp(ir \rho(P_j))
=\sum_{j=0}^{n-1} e^{ir_j} \rho(P_j). 
\]

\begin{claim} \label{C.trace}
If  $A$ is isomorphic to  
$M_k(\bbC)$ for some natural number $k$ then~$\Theta$ preserves the normalized 
trace, $\Tr$, of the unitaries. 
\end{claim} 

\begin{proof} 
By the above computation, we only need to show there is $d\in \bbN$
such that for every $m$, every projection of rank $m$  in $\cA$ is mapped to a projection of 
rank $dm$ in $\cB$. But this is an immediate consequence of
Claim~\ref{C.MvN} and the obvious
equality $\rho(P+Q)=\rho(P)+\rho(Q)$\footnote{This equality  holds
for any two self-adjoint operators, but we shall not need it.}
 for commuting projections $P$ and $Q$,
since in $M_k(\bbC)$ two projections are Murray-von Neumann
equivalent if and only if they have the same rank. 
\end{proof}

\begin{claim} The map 
 $\Upsilon\colon \cA \to \cB$ given by $\Upsilon(\sum_j \alpha_j
u_j)=\sum_j \alpha_j \Theta(u_j)$ whenever $\alpha_j$ are scalars and $u_j$ are
unitaries is a well-defined *-homomorphism from $\cA$ into $\cB$. 
\end{claim} 

\begin{proof}
Since every
operator in $\cA$ is a linear combination of four unitaries (cf. the proof of
Lemma~\ref{L.L.1}) 
 in order to see that $\Upsilon$ is well defined
  we only need to check that $\sum_j \alpha_j u_j=0$ implies
$\sum_j \alpha_j \Theta(u_j)=0$.

Let us first consider the case when $\cA$ is a full matrix algebra. 
 The following argument is taken from Dye (\cite[Lemma~3.1]{Dye:Unitary}).

 Assume  $a=\sum_i \alpha_i u_i=0$. Then
$0=\Tr(aa^*)=\sum_{i,j}\alpha_i\bar\alpha_j \Tr(u_iu_j^*)$. Also with $b=\sum_i
\alpha_i \Theta(u_i)$ we have
$$
\textstyle\Tr(bb^*) =\sum_{i,j} \alpha_i
\bar\alpha_j\Tr(\Theta(u_iu_j^*))=\sum_{i,j}\alpha_i\bar\alpha_j\Tr(u_iu_j^*),
$$
which is 0 by Claim~\ref{C.trace}. Therefore $b=0$, proving that $\Upsilon$ is
well-defined when $\cA$ is a full matrix algebra. 

In order to prove the general case, 
 let $S_0,\dots, S_{m-1}$ list all minimal central projections of $\cA$. 
Then $S_i \cA S_i$ is isomorphic to some $M_{k(i)}(\bbC)$ and
therefore $\Upsilon$ is well-defined on this subalgebra. 
However, $\Theta(u)=\sum_{j=0}^{m-1} \rho(S_j) \Theta(u)$ for all unitaries $u$ in $\cA$, 
and therefore $\Upsilon (a)=\sum_{j=0}^{m-1}\rho(S_j) \Upsilon(a)$ is well-defined for all $a\in \cA$. 

Clearly $\Upsilon$ is a complex vector space homomorphism and
$\Upsilon(u)=\Theta(u)$ for a unitary $u$ in $\cA$. It is
straightforward to check that $\Upsilon$ is multiplicative and a
*-homomorphism.
\end{proof} 

Any $a\in \cA_{\leq 1}$ can be written as $b+ic$, where $b$ and $c$
are self-adjoints of norm $\leq 1$, and $\Lambda(a)\approx_{3\e}
\Lambda(b)+i\Lambda(c)$. If $b$ is self-adjoint of norm $\leq 1$, then there
are unitaries $u$ and $v$ such that $b=u+v$ (cf. the proof of
Lemma~\ref{L.L.1}). Therefore $\Lambda(b)\approx_\e
\Lambda(u)+\Lambda(v)\approx_{2\e_3} \Xi(u)+\Xi(v)=\Upsilon(b)$. All in all, we
have $\|\Lambda(a)-\Upsilon(a)\|\leq \e+\e_3$ for $a\in \cA_{\leq
1}$.
 Since for small $\e$ each $\e_i$ for $1\leq i\leq 3$ is bounded by a linear function of $\e$,
 this concludes the proof.
 \end{proof}

The assumption that $\Lambda$ is unital was not necessary in
Theorem~\ref{T.approximate}. This is an easy consequence of
Theorem~\ref{T.approximate} and the following well-known lemma, whose
proof can also be found in \cite{Lor:Lifting}

\begin{lemma} \lbl{L.approximate.projection} If $0<\e<1/8$ then
in every C*-algebra $\cA$ the following holds. For every $a$ satisfying
$\|a^2-a\|<\e$ and $\|a^*-a\|<\e$ there is a projection~$Q$ such that
$\|a-Q\|<4\e$.
\end{lemma}

\begin{proof}
 We claim that
$M=\|a\|<2$. Let $b=(a+a^*)/2$. Then $\|a-b\|<\e/2$,  $b$ is self-adjoint, and
 $\|b\|>M-\e/2$. Consider $\cA$ as a concrete C*-algebra acting on some
Hilbert space $H$. In the weak closure of $\cA$ find a spectral projection $R$
of $b$ corresponding to $(M-\e/2+\|a-b\|,\|b\|]$. If $\xi$ is a unit vector in
the range of $R$, then $\|b\xi-M\xi\|<\e/2-\|a-b\|$. If $\eta=a\xi-M\xi$ then
$\|\eta\|<\e/2$, and $M-\e/2<\|a\xi\|\leq M$. Also,
$a^2\xi=a\eta+Ma\xi=a\eta+M\eta+M^2\xi$, and therefore $\|a^2\xi\|\geq
M^2-M\e$. Since  $\|a^2\xi-a\xi\|<\e$, we have
$\|a^2\xi\|<\|a\xi\|+\e<M+3\e/2$. Therefore  $M+3\e/2>M^2-M\e$, and with
$\e<1/4$ this implies $M<2$ as claimed.

Therefore  $\|aa^*-a\|<2\|a^*-a\|+\|a^2-a\|<3\e$. So we have
\begin{equation}
4\|b^2-b\| =\|a^2+aa^*+a^*a+(a^*)^2-2a-2a^*\|<8\e.\lbl{E.4b}
\end{equation}
We may assume $\cA$ is unital. Since $b$ is self-adjoint, via the function
calculus in $C^*(b,I)$, the subalgebra of $\cA$ generated by $b$ and $I$, $b$
 corresponds to the identity function on its spectrum $\sigma(b)$. By
 \eqref{E.4b},
 for every $x\in \sigma(b)$ we have $|x^2-x|<2\e$. Therefore $1/2\notin
 \sigma(b)$,
  $U=\{x\in \sigma(b)\mid |x-1|<1/2\}$ is a relatively closed and open subset of
 $\sigma(b)$, and the projection $Q$ corresponding to the
 characteristic function of this set in $C^*(b,I)$ belongs to $\cA$.
Then $\|Q-b\|=\sup_{x\in \sigma(b)}\min(|x|,|1-x|)$. If
$\delta(\e)=\sup\{\min(|x|,|1-x|)\mid |x^2-x|<2\e\}+\e/2$, then
$\|a-Q\|<\delta(\e)$. Clearly $\delta(\e)<4\e$ for $\e<1$. \end{proof}

\section{Automorphisms with C-measurable representations are inner}
\lbl{S.Borel}

Each known proof that all automorphisms of a quotient structure related to
$\cP(\bbN)/\Fin$ or $\cB(H)/\cK(H)$ are `trivial' proceeds in two stages. In
the first some additional set-theoretic axioms are used to prove that all
automorphisms are `topologically simple.' In the second all `topologically
simple' automorphisms are shown to be trivial, without using any additional
set-theoretic axioms (see \cite{Fa:Rigidity}).
 The present proof is not an exception and the present section deals with the
 second step.  Even though no additional set-theoretic axioms
are needed for its conclusion, the proof of Theorem~\ref{T.3} given
at the end of this section will take a metamathematical detour via TA
and Shoenfield's theorem (Theorem~\ref{T.Shoenfield}). Note that the
latter is not needed for the proof of Theorem~\ref{T.0}, since TA
follows from its assumptions.

\begin{thm}\lbl{T.3} Every automorphism of $\cC(H)$
with a C-measurable representation on $\cU(H)$ is inner.
\end{thm}

\subsection{Inner on FDD von Neumann algebras}
 If $v$ is a linear isometry between
cofinite-dimensional subspaces of $H$ then $\Psi_v(a)=vav^*$ is a
representation of an automorphism of $\cC(H)$. We use notation $\vec
E$, $\calD[\vec E]$ and $\calD_M[\vec E]$ from \S\ref{S.Notation}.

\begin{lemma}\lbl{L.trivial.2}
Assume $\# E_n$ is a nondecreasing sequence. If an automorphism
$\Phi$ of the Calkin algebra is inner on $\calD_M[\vec E]$ for some
infinite $M$, then it is inner on $\calD[\vec E]$.
\end{lemma}

\begin{proof} By Lemma~\ref{L.trivial}, it will suffice to find a partial isometry $u$ such that
$u\calD[\vec E]u^*\subseteq \calD_M[\vec E]$ and $u^*u=I$. If $(m_j)$
is an increasing enumeration of $M$, then $\# E_j\leq \# E_{m_j}$ by
our assumption. Let $u_j\colon \SPAN\{e_i\mid i\in E_j\}\to
\SPAN\{e_i\mid i\in E_{m_j}\}$  be a partial isometry. Then $u=\sum_j
u_j$ is as required.
\end{proof}

\begin{thm} \lbl{T.2}Assume $\Phi$ is an automorphism of $\cC(H)$,  $\vec E$
is a partition of~$\bbN$ into finite intervals, and
   $\Phi$  has a C-measurable representation on $\calD[\vec E]$.
Then $\Phi$ has a representation which is a *-ho\-mo\-mor\-phi\-sm from $\calD[\vec
E]$ into~$\cB(H)$. Moreover,  there is a partial isomorphism~$v$ of
cofinite-dimensional subspaces of $H$ such that $\Psi_v$ is a representation
of~$\Phi$ on $\calD[\vec E]$. 
\end{thm}

\begin{proof}  By coarsening $\vec E$ we may assume the sequence $\# E_n$ is nondecreasing. Since $\vec E$ is fixed, we
write $\bP_A$ for $\bP^{\vec E}_A$. The proof proceeds by successively
constructing a sequence of representations, each one with more adequate
properties than the previous ones, until we reach one that is a *-homomorphism
between the underlying algebras. This is similar to the proofs in
\cite[\S1]{Fa:AQ}. Some of the arguments may also resemble those from
\cite{Arv:Notes}.

Let $\e_i=2^{-i}$. Fix a finite $\e_i$-dense in norm subset $\ba_i\subseteq
\cB(\SPAN\{e_i\mid i\in E_n\})_{\leq 1}$ containing the identity and zero. Note
that $\prod_{j=l}^{l+m}\ba_j$ is $2\e_l$-dense in
$\prod_{j=l}^{l+m}\cB(\SPAN\{e_i\mid i\in E_j\})_{\leq 1}$. Let $\bA=\prod_i
\ba_i$. We shall identify $a\in \ba_i$ with $\bar a\in \calD[\vec E]$ such that
$\bP_{\{i\}}\bar a=a$ and $(I-\bP_{\{i\}})\bar a=0$. For $J\subseteq \bbN$ and
$x\in \bA$ it will be convenient to write
 $x\rs J$ for the projection of $x$ to
 $\prod_{i\in J} \ba_i$, identified with  $\bP_Jx$.

\begin{claim}  \lbl{C.1} There is a strongly continuous representation $\Psi_1$ of $\Phi$ on $\bA$.
\end{claim}

\begin{proof}
 Since each $\ba_i$ is finite, the strong operator
topology on $\bA$ coincides with its Cantor-set topology which is compact
metric. Let $\cX\subseteq \bA$ be a dense $G_\delta$ set on which $\Psi$ is
continuous. Write $\cX$ as an intersection of dense open sets $U_n$, $n\in
\bbN$. Since each $\ba_i$ is finite, a straightforward diagonalization argument
produces an increasing sequence $(n_i)$ in $\bbN$, with $J_i=[n_i,n_{i+1})$,
$\bb_i=\ba_{J_i}=\prod_{k\in J_i} \ba_i$ and  $s_i\in \bb_{i}$ such that for
all $x\in \bA$ and all $i$ we have $x\rs J_i=s_i$ implies $x\in \bigcap_{j=0}^i
U_j$. Therefore $\{x\mid (\exists^\infty i) x\rs J_i=s_i\}\subseteq \cX$.

Let $C_0=\bigcup_{j\ \even} J_j$, $C_1=\bigcup_{j\ \odd}J_j$, $R_0=\bP_{C_0}$
and $R_1=\bP_{C_1}$,  let $S_0=\sum_{j\ \odd} s_j$ and let $S_1=\sum_{j\ \even}
s_j$. Note that $R_iu=uR_i=R_iuR_i$ for all $u\in \calD[\vec E]$ and $i\in
\{0,1\}$. For $u\in \bA$ let
\begin{enumerate}
\item [(*)] $\Psi_1(u)=\Psi(uR_0+S_0)-\Psi(S_0)+\Psi(uR_1+S_1)-\Psi(S_1)$.
\end{enumerate}
Then $\Psi_1$ is a continuous representation of $\Phi$ on $\bA$.
\end{proof}

Our next task is to find a representation $\Psi_2$ of  $\Phi$ on $\bA$ which is
stabilized (in a sense to be made very precise below) and then extend it to a
representation of $\Phi$ on $\calD[\vec E]$.
Start with $\Psi_1$ as provided by
Claim~\ref{C.1}. By possibly replacing $\Psi_1$ with
 $b\mapsto \Psi_1(b)\Psi_1(I)^*$, we may
assume $\Psi_1(I)=I$.

The sequence of projections $(\bfR_k)$ was fixed in \S\ref{S.Notation}.

\begin{claim}\lbl{C.stabilizer.0}
For all $n$ and $\e>0$ there are  $k>n$ and $u\in \prod_{i=n}^{k-1} \ba_i$ such
that for all $a$ and $b$ in $\bA$ satisfying  $a\rs [n,\infty)=b\rs [n,\infty)$
and $a\rs [n,k)=u$ we have
\begin{enumerate}
\item  $\|(\Psi_1(a)-\Psi_1(b))(I-\bfR_{k})\|\leq \e$ and
\item $\|(I-\bfR_{k})(\Psi_1(a)-\Psi_1(b))\|\leq \e$.
\end{enumerate}
\end{claim}

\begin{proof} Write $\bc=\prod_{i=0}^{n-1}\ba_i$. For $a\in \bA$ and $s\in \bc$
write $a[s]=s+\bP_{[n,\infty)}a$.
 For $k>n$ let
 \begin{align*}
 V_k=\{a\in \bA\mid (\exists s\in \bc)(\exists t\in \bc)
&\|(\Psi_1(a[s])-\Psi_1(a[t]))(I-\bfR_{k})\|>\e\\
\text{ or } &\|(I-\bfR_{k})(\Psi_1(a[s])-\Psi_1(a[t]))\|>\e\}.
\end{align*}
Since $\Psi_1$ is continuous,  each $V_k$ is an open subset of $\bA$.
 If $a\in \calD[\vec E]$, and $s$ and $t$ are in $\bc$ then
$\Psi_1(a[s])-\Psi_1(a[t])$ is compact and therefore
$\|(\Psi_1(a[s])-\Psi_1(a[t]))(I-\bfR_{k})\|\leq \e$ and
$\|(I-\bfR_{k})(\Psi_1(a[s])-\Psi_1(a[t]))\|\leq \e$ for a large
enough $k=k(a,s,t)$. Since $\bc$ is finite, for some large enough
$k=k(a)$ we have $a\notin V_k$. Therefore the $G_\delta$ set
$\bigcap_k V_k$ is empty. By the Baire Category Theorem, we may  fix
$l$ such that $V_l$ is not dense. There is a basic open set disjoint
from $V_l$. Since $a\in V_l$ if and only if $a[s]\in V_l$ for all $a$
and $s\in \bc$, for some $k\geq l$ there is
 a $u\in \prod_{i=n}^{k-1}\ba_i$ such that $\{a\in \bA\mid a\rs [n,k)=u\}$ is
 disjoint from $V_k$ (note that $V_k\subseteq V_l$).
Then $k$ and $u$ are as required.
\end{proof}

We shall  find two increasing sequences of natural numbers, $(n_i)$ (unrelated
to the one appearing in the proof of Claim~\ref{C.1}) and $(k_i)$ so that
$n_i<k_i<n_{i+1}$ for all $i$. These sequences will be chosen according to the
requirements described below. With $J_i=[n_i,n_{i+1})$ write
$\bb_i=\ba_{J_i}=\prod_{j\in J_i} \ba_j$.

Let $\e_i=2^{-i}$.  A $u_i\in \bb_i$ is an \emph{$\e_i$-stabilizer for
$\Psi_1$} (or a  \emph{stabilizer}) if for all $a,b$ in $\bA$ such that $a\rs
[n_i,n_{i+1})=b\rs [n_i,n_{i+1})=u_i$ the following hold.
\begin{enumerate}
\item [(a)] If  $a\rs [n_i,\infty)=b\rs [n_i,\infty)$ then
\begin{enumerate}
\item [(a1)] $\|(\Psi_1(a)-\Psi_1(b))(I-\bfR_{k_i})\|<\e_i$ and
\item[(a2)] $\|(I-\bfR_{k_i})(\Psi_1(a)-\Psi_1(b))\|<\e_i$.
\end{enumerate}
\item [(b)] If $a\rs [0,n_{i+1})=b\rs [0,n_{i+1})$ then
\begin{enumerate}
\item [(b1)] $\|(\Psi_1(a)-\Psi_1(b))\bfR_{k_i}\|<\e_i$ and
\item [(b2)] $\|\bfR_{k_i}(\Psi_1(a)-\Psi_1(b))\|<\e_i$.
\end{enumerate}
\end{enumerate}
We shall find $(n_i)$, $(k_i)$, $J_i$, $\bb_i$ as above and a stabilizer
$u_i\in\bb_i$ for all $i$. Assume all of these objects up to and including
$n_i$, $k_{i-1}$ and $u_{i-1}$ have been chosen to satisfy the requirements.
Applying Claim~\ref{C.stabilizer.0}, find $k_i\geq n_i$ and $u_i^0\in
\prod_{j=n_i}^{k_i-1}\ba_j$ such that (a1) and (a2) hold. Then apply the
continuity of $\Psi_1$ to find $n_{i+1}\geq k_i$ and $u_i\in
\prod_{j=n_i}^{n_{i+1}-1}\ba_j$  such that $u_i\rs [n_i,k_i)=u_i^0$ and  (b1)
and (b2) hold as well.

 Once the sequences
$n_{i+1}, k_i$ and  $u_i\in \bb_i=\prod_{j\in J_i}\ba_j$ are chosen,  let
$$
\cV_i=\textstyle\bigoplus_{j\in J_i}\cB(E_j).
$$
 Then $\calD[\vec E]=\prod_i
\cV_i$. We identify $\cV_j$ with $\bP_{J_i}\calD[\vec E]$ and
 $b\in \calD[\vec E]_{\leq 1}$ with the sequence $\langle b_j\rangle_j$
such that $b_j\in \cV_j$ and  $b=\sum_j b_j$. Let $I_j$ denote the
identity of $\cV_j$. Note that $I_j\in \bb_j$. Recall that $\bb_i$ is
$2\e_i$-dense in $(\cV_i)_{\leq 1}$ and fix a linear ordering of each
$\bb_i$. Define
$$
\sigma_i\colon \cV_i\to \bb_i
$$
by letting $\sigma_i(c)$ be the first
element of $\bb_i$ that is within $2\e_i$ of $c$. For $c\in \calD[\vec E]_{\leq
1}$ let 
\[
\textstyle c_{\even}=\sum_i \sigma_{2i}(c_{2i})\qquad\text{and}\qquad c_{\odd}=\sum_i
\sigma_{2i+1}(c_{2i+1}). 
\]
Both of these elements belong to $\bA$ and $c-c_{\even}
-c_{\odd}$ is compact.

Let us concentrate on $\cV_{2i+1}$. Define $\Lambda_{2i+1}\colon \cV_{2i+1}\to
\cB(H)$:
$$
\Lambda_{2i+1}(b)=\Psi_1(u_{\even} +\sigma_{2i+1}(b))-\Psi_1(u_{\even}).
$$
Since both $\sigma_i$ and $\Psi$ are  Borel-measurable, $\Lambda_{2i+1}$ is
Borel-measurable as well.  Let $\bfQ_i=\bfR_{k_{i+1}}-\bfR_{k_{i-1}}$, with
$k_{-1}=0$.

\begin{claim} \lbl{C.b-cpct}
For $b\in \calD[\vec E]_{\leq 1}$
 such that $b_{2i}=0$ for all $i$ the operator
$\Psi_1(b)-\sum_{i=0}^\infty \bfQ_{2i+1}\Lambda_{2i+1}(b_{2i+1})\bfQ_{2i+1}$ is
compact. In particular the latter operator is bounded. 
\end{claim}

\begin{proof}
Since $b-b_{\odd}$  is compact, so is
$\Psi_1(b)+\Psi_1(u_{\even})-\Psi_1(u_{\even} +b_{\odd})$. By applying (a1) and (b1)
to $b_{\odd}$, $b^+=\sum_{j=i}^\infty \sigma_{2j+1}(b)$ and $\sigma_{2i+1}(b_{2i+1})$ we see that
\begin{align*}
\|(\Lambda_{2i+1}(b_{2i+1})+\Psi_1&(u_{\even})-\Psi_1(u_{\even}
+b_{\odd}))\bfQ_{2i+1}\|\\
=&\|\Psi_1((u_{\even}+\sigma_{2i+1}(b))-\Psi_1(u_{\even}+b_{\odd}))\bfQ_{2i+1})\|\\
\leq &\|(\Psi_1(u_{\even}+\sigma_{2i+1}(b))-\Psi_1(u_{\even}+b^+))\bfQ_{2i+1}\|\\
& +
\|(\Psi_1(u_{\even}+b^+)-\Psi_1(u_{\even}+b_{\odd}))\bfQ_{2i+1}\|\\
<&2\e_{2i+1}.
\end{align*}
 Since $\sum_i(\e_{2i+1})^2<\infty$ and $I-\sum_i \bfQ_{2i+1}$
 is a compact operator, the operator
$\Psi_1(u_{\even} +b_{\odd})-\Psi_1(u_{\even})-\sum_{i=0}^\infty
\Lambda_{2i+1}(b_{2i+1})\bfQ_{2i+1}$ is compact. An analogous proof using (a2)
and (b2) instead of (a1) and (b1) gives that $\Psi_1(u_{\even}
+b_{\odd})-\Psi_1(u_{\even})-\sum_{i=0}^\infty
\bfQ_{2i+1}\Lambda_{2i+1}(b_{2i+1})\bfQ_{2i+1}$ is compact.
\end{proof}

Define $\Lambda_{2i+1}'\colon \cV_{2i+1}\to \cB(H)$ by
$$
\Lambda_{2i+1}'(b)=\bfQ_{2i+1}\Lambda_{2i+1}(b)\bfQ_{2i+1}.
$$
With $a_{2i+1}=\Lambda_{2i+1}'(I_{2i+1})$ let $
\e_i=\max(\|a_{2i+1}^2-a_{2i+1}\|,\|a_{2i+1}^*-a_{2i+1}\|)$.  We claim that
$\limsup_i \e_i=0$. Assume not and find $\e>0$ and an infinite $M\subseteq
2\bbN+1$ such that for all $i\in M$ we have
$\max(\|a_i^2-a_i\|,\|a_i^*-a_i\|)>\e$.    With
 $a=\sum_{i\in M} a_i$ the operator
$\Psi_1(\bP_M)-a$ is compact, thus $a^*-a$ and $a^2-a$ are both compact. Since
$a_i=\bQ_ia\bQ_i$ and $\bQ_i\bQ_j=0$ for distinct $i$ and $j$ in $M$, we have
$a^*=\sum_{i\in M} a_i^*$ and $a^2=\sum_{i\in M} a_i^2$. By the choice of $M$
and $\e$ at least one of $a-a^*$ and $a^2-a$  is not compact, a contradiction.

Applying Lemma~\ref{L.approximate.projection} to $a_{2i+1}$ such that $\e_{i}$
is small enough, obtain projections $\bfS_{2i+1}\leq \bQ_{2i+1}$ such that
$\limsup_{i\to\infty} \|\bfS_{2i+1}-\Lambda_{2i+1}'(I_{2i+1})\|=0$. With
Lemma~\ref{L.trivial} in mind, we shall ignore all the even-numbered
$\cV_i$~and~$\Lambda_i$. Let
$$
\Lambda_{i}''(a)=\bfS_{2i+1}\Lambda_{2i+1}'(a)\bfS_{2i+1}
$$
for $a\in \cV_{2i+1}$ and let $\bfS''_i=\bfS_{2i+1}$ and $\cV''_i=\cV_{2i+1}$
for all $i$.

Then $ \Lambda''(a)=\sum_i  \Lambda''_i(a_i)$ is a representation of $\Phi$ on
$\bigoplus_i \cV''_{i}$.
  For $j\in \bbN$ let
 \begin{align*}
 \delta_j^0&=\sup_{a,b\in (\cV''_j)_{\leq
 1}}\| \Lambda''_j(ab)- \Lambda''_j(a) \Lambda''_j(b)\|,\\
 \delta_j^1&=\sup_{a,b\in (\cV''_j)_{\leq
 1}}\| \Lambda''_j(a+b)- \Lambda''_j(a)- \Lambda''_j(b)\|,\\
   \delta_j^2&=\sup_{a\in (\cV''_j)_{\leq
 1}}\| \Lambda''_j(a^*)- \Lambda''_j(a)^*\|,\\
\delta_j^4&=\sup_{a\in (\cV''_j)_{\leq
 1}} |\|a\|-\| \Lambda''_j(a)\||.
\end{align*}
 We claim that   $\lim_j \max_{0\leq k\leq 4}\delta_j^k=0$.
 We shall prove only $\lim_j \delta_j^0=0$
 since the other proofs are similar.
Assume the limit is nonzero, and for each $j$ fix  $b_j$ and $c_j$ in
$(\cV_{j}'')_{\leq 1}$ such that $\| \Lambda''_{j}(b_jc_j)- \Lambda''_{j}(b_j)
\Lambda''_{j}(c_j)\|\geq \delta^0_j/2$ for all~$j$. Let $b$ and $c$ in
$\cB[\vec E]_{\leq 1}$ be such that $\bP_{J_j}b=b_j$  and $\bP_{J_j}c=c_j$ for
all $j$. Then  $\Psi_1(bc)-\Psi_1(b)\Psi_1(c)$ is compact. By
Claim~\ref{C.b-cpct}, so is $\sum_j
 \Lambda''_{j}(b_jc_j)- \Lambda''_{j}(b_j) \Lambda''_{j}(c_j)$.
This implies  $\lim_j \delta^0_j=0$, a contradiction.

Each  $ \Lambda_j''$ is a $2\delta_j$-approximate
*-homomorphism as defined in \S\ref{S.approximate}. Since   $\lim_j 2\delta_j=0$ and each $ \Lambda''_j$ is
Borel-measurable, by applying Theorem~\ref{T.approximate} to $
\Lambda_j''$ for $j$ larger than some $n_0$  we find a
$2K\delta_j$-approximation to $ \Lambda_j''$ which is a unital
*-homomorphism, $\Xi_i\colon \calD_{2i+1}\to \cB( \bfS''_i[H])$. For
 $i\leq n_0$ let $\Xi_i$ be identically equal to 0.
Since $\lim_j 2K\delta_j=0$ and $ \bfS''_i$ are pairwise orthogonal, the
diagonal $\Xi$ of $\Xi_i$  is a
*-homomorphism and a representation of $\Phi$ on $\calD_{\bigcup_{i\text{ odd}}J_i}[\vec E]$.

Still ignoring the even-numbered $\cV_j$'s, we address the second part of
Theorem~\ref{T.2} by showing $\Phi$ is inner on $\calD_{\bigcup_{i\text{
odd}}J_i}[\vec E]$.  Let $F_i=\bP_{J_i}[H]$ and $G_i=\bfS''_i[H]$.

\begin{claim} For all but finitely many $i$ there is a linear
isometry $v_i\colon F_i\to G_i$ such that $\Xi_i(a)=v_iav_i^*$ for all $a\in
\calD[(E_j)_{j\in J_i}]$.
\end{claim}

\begin{proof} Let $\xi_n$, for $n\in \bbN$, 
 be an orthonormal sequence such that
each $\xi_n$ belongs to some $F_i$ and no two $\xi_n$ belong to the same $F_i$.
Let $P=\proj_{\SPAN\{\xi_n\mid n\in \bbN\}}$ and consider the masa $\cA$ of
$\cB(P[H])$ consisting of all operators diagonalized by $\xi_n$, for $n\in \bbN$. The image
under the quotient map of $\cA$ in the Calkin algebra $\cC(P[H])$  is a masa
(\cite{JohPar}). It is contained in the domain of $\Xi$. The image of the
$\Xi$-image of $\cA$  is a masa in  $\cC(\Xi(P)[H])$. Because of this, for all
but finitely many $n$ the projection $\Xi(\proj_{\bbC\xi_n})$ has rank $1$.
Since $(\xi_n)$ was arbitrary, for all but finitely many $n$ and all
one-dimensional projections $R\leq \proj_{F_n}$ the rank of $\Xi(R)$ is equal
to 1. Fix such $n$ and a basis $(\eta_j\mid j<\dim(E_n))$ of $F_n$. Let
$P_j=\Xi(\proj_{\bbC\eta_j})$. For all but finitely many $n$ we have
$\sum_{j<\dim(F_n)}P_j=\proj_{G_n}$. Consider $n$ large enough for this to
hold. Fix a unit vector $\xi_0$ in the range of $P_0$. Let $a\in \cU(F_n)$ be
generated by a cyclic permutation of $\{\eta_j\}$, so that
$a(\eta_j)=\eta_{j+1}$ (with $\eta_{\dim(F_n)}=\eta_0$). With $b=\Xi(a)$ let
$\xi_j=b^j(\xi_0)$ (here $b^j$ is the $j$-th power of $b$). Then $(\xi_j)$ form
a basis of $G_n$. It is clear that $\eta_j\mapsto \xi_j$ defines an isometry
$v_n$ as required.
\end{proof}

For a large enough $m$ the sum $v=\bigoplus_{n=m}^\infty v_n$ is a partial
isometry from $\bigoplus_{n=m}^\infty F_n$ to $\bigoplus_{n=m}^\infty G_n$ such
that $\Xi(a)-vav^*$ has finite rank for all $a\in \calD[\vec E]$. Lemma~\ref{L.trivial.2} implies $\Phi$ is
inner on $\calD[\vec E]$.
\end{proof}

\begin{proof}[Proof of Theorem~\ref{T.3}]  Fix an automorphism $\Phi$ of  $\cC(H)$ with
a C-measurable representation.
 By Lemma~\ref{L.guess} (2)  we may assume
 $\Phi$ has a Borel-measurable representation $\Psi$.
 Let $B\subseteq \cB(H)_{\leq 1}\times \cB(H)_{\leq
1}$ be the set of all pairs $(a,b)$ such that $\Psi(b)-aba^*$ is not
compact. Then the assertion of Theorem~\ref{T.3} is equivalent to
$(\exists a)(\forall b)(a,b)\notin B$.  Lemma~\ref{L.T.3} implies $B$
is Borel and therefore by Theorem~\ref{T.Shoenfield} we may use TA in
the proof of Theorem~\ref{T.3}.

By Theorem~\ref{T.2}, $\Phi$ is inner on $\calD[\vec E]$ for each
finite-dimensional decomposition $\vec E$ of $H$. By Theorem~\ref{P.1},
$\Phi$ is inner.
\end{proof}

\section{Locally inner automorphisms}
\lbl{S.Local}

Fix an automorphism $\Phi$ of $\cC(H)$. Proposition~\ref{P.2} below  is roughly
mo\-de\-led on the proof of \cite[Proposition~3.12.1]{Fa:AQ}. 
Its main components are Lemma~\ref{L.cJ-m},  
Proposition~\ref{L.sigma}, and Theorem~\ref{T.2}. 
The key device in the proof of Lemma~\ref{L.cJ-m} is the partition defined
in (K1)--(K3). It is a descendant of Velickovic's partition (\cite{Ve:OCA}) 
and the partitions used in \cite[p. 100]{Fa:AQ}. 

If $u$ is a
partial isomorphism we write $\Psi_u$ for the conjugation, $\Psi_u(a)=uau^*$.
Fix a partition $\vec E$ of $\bbN$ into finite intervals such that the sequence
$\# E_n$ is nondecreasing.

\begin{prop} \lbl{P.2} TA implies
$\Phi$ is inner on $\calD[\vec E]$.
\end{prop}

Using the Axiom of Choice, find a representation $\Psi\colon \cB(H)\to \cB(H)$
of $\Phi$. It is not assumed that $\Psi$ is C-measurable or that it is a
homomorphism, but we may assume $\Psi(P)$ is a projection whenever $P$ is a
projection. This is because every projection in the Calkin algebra is the image
of some projection in $\cB(H)$ via the quotient map
  (\cite[Lemma~3.1]{We:Set}).  We may also assume
$\|\Psi(a)\|\leq \|a\|$ for all $a$: find the polar decomposition of
$\Psi(a)$, apply the spectral theorem to its positive part, and
truncate the function to~$\|a\|$.

For $M\subseteq \bbN$ let  $\cU_M[\vec E]$ denote the unitary group of  $\calD_M[\vec E]$ and let
 \begin{align*}
 \cJ^n(\vec E)=\{M\subseteq \bbN\mid &\text{there is a Borel-measurable $\Xi\colon
\cU_M[\vec E]\to \cB(H)$}\\
&(\forall a\in \cU_M[\vec E])\|\Phi(\pi(a))-\pi(\Xi(a))\|\leq 2^{-n}\}, \\
\cJ^n_\sigma(\vec E)=\{M\subseteq \bbN\mid &\text{ there are Borel-measurable
 $\Psi_i\colon \cU_M[\vec E]\to \cB(H)$, $i\in \bbN$}\\
 &(\forall a\in \cU_M[\vec E])(\exists i)
\|\Phi(\pi(a))-\pi(\Psi_i(a))\|\leq 2^{-n}\}.
\end{align*}
In the terminology of \S\ref{S.epsilon}, $\Xi$ is a $2^{-n}$-approximation to
$\Phi$ on $\cU_M[\vec E]$.
 Each $\cJ^n(\vec E)$ and each $\cJ^n_\sigma(\vec E)$ is
hereditary and closed under finite changes of its elements, but these sets are
not necessarily closed under finite unions.

 Given $\vec E=(E_n)_{n=0}^\infty$
write $F_n=\SPAN\{e_i\mid i\in E_n\}$ and $\bP^{\vec E}_A$ for the projection
to $\bigoplus_{n\in A} F_n$. While $\vec E$ is fixed we shall drop the
superscript and write~$\bP_A$.
 A family of
subsets of $\bbN$ is \emph{almost disjoint} if $A\cap B$ is finite for all
distinct $A$ and $B$ in the family.
 An almost disjoint family $\cA$ is \emph{tree-like} if there is a partial ordering $\preceq$
of $\bbN$ such that $(\bbN,\preceq)$ is isomorphic to $(2^{<\bbN},\subseteq)$
and each element of $\cA$ is a maximal branch of this tree. If $J_s$ ($s\in
2^{<\bbN}$) are pairwise disjoint finite subsets of $\bbN$ and $X\subseteq
2^{\bbN}$, then the family of all $M_x=\bigcup_n J_{x\rs n}$, $x\in X$, is
tree-like, and every tree-like family is of this form.

\begin{lemma} \lbl{L.cJ-m}  TA implies that for every $k$
every tree-like family of $\cJ^k_\sigma(\vec E)$-positive sets is at most
countable.
\end{lemma}

\begin{proof}
Fix an uncountable tree-like family $\cA$ and a partial  ordering
$\preceq$ on $\bbN$ such that $(\bbN,\preceq)$ is isomorphic to
$(2^{<\bbN},\subseteq)$ and all elements of $\cA$ are maximal
branches in $(\bbN, \preceq)$.  Let
$$
\cX=\{(S,a)\mid\text{$S$ is infinite and } (\exists
B(S)\in\cA)(S\subseteq B(S) \text{ and } a\in \cU_S[\vec E])\}.
$$
Note that $(S,a)\in \cX$ implies $\bP_Sa=a\bP_S=\bP_Sa\bP_S=a$. Also,
for $i\in S$  we have that $\bP_{\{i\}}a\in \cB(F_i)$. If moreover
$(T,b)\in \cX$, then $\bP_S\bP_T=\bP_{S\cap T}$  and for each $i$ we
have $(a-b)\bP_{\{i\}}=\bP_{\{i\}}(a-b)=\bP_{\{i\}}(a-b)\bP_{\{i\}}$.

Modify $\Psi$ as follows. If $a\in \calD_B[\vec E]\setminus\cK(H)$
for some $B\in \cA$ then replace $\Psi(a)$ with
$\Psi(\bP_B)\Psi(a)\Psi(\bP_B)$. Since $a$ is not compact such $B$ is
unique and since $\bP_Ba\bP_B=a$ the modified $\Psi$ is a
representation of $\Phi$ which satisfies $\|\Psi(a)\|\leq \|a\|$ for
all $a$ and $\Psi(a)\Psi(\bP_B)=\Psi(\bP_B)\Psi(a)$ for $a$ and $B$
as above.

Fix $n\in \bbN$. Define a partition $[\cX]^2=K^{n}_0\cup K^{n}_1$ by
letting $\{(S,a),(T,b)\}$ in $K^{n}_0$ if and only if the following
three conditions hold
\begin{enumerate}
\item [(K1)] $B(S)\neq B(T)$,
\item [(K2)] for each $i\in S\cap T$  we have $\|(a-b)\bP_{\{i\}}\|<2^{-i}$.
\item [(K3)] $\|\Psi(a)\Psi(\bP_T)-\Psi(\bP_S)\Psi(b)\|>2^{-n}$ or\\
 $\|\Psi(\bP_T)\Psi(a)-\Psi(b)\Psi(\bP_S)\|>2^{-n}$.
\end{enumerate}
The definition is clearly symmetric.  Consider $\cP(\bbN)$ with
the Cantor-set topology (\S\ref{S.P(N)}) and $\cB(H)_{\leq 1}$ with the strong
operator topology.

\begin{claim} \lbl{C0} The coloring  $K^{n}_0$ is open in the
topology on $\cX$ obtained by identifying  $(S,a)$ with
$(B(S),S,a,\Psi(\bP_S),\Psi(a))\in \cP(\bbN)^2\times (\cB(H)_{\leq 1})^3$.
\end{claim}

\begin{proof}  Assume the pair $(S,a),(T,b)$ satisfies (K1). Since $S$ and $T$
are infinite subsets of disjoint branches of $(\bbN,\preceq)$, their
intersection is finite and we may fix  $s\in S\cap (B(S)\setminus B(T))$ and
$t\in T\cap (B(T)\setminus B(S))$. Then $U=\{(S',a')\mid s\in S'\}$ and
$V=\{(T',b')\mid s\in T'\}$ are open neighborhoods of $(S,a)$ and $(T,b)$ and
any pair in $U\times V$ satisfies (K1).

 We shall show (K2) is open relatively to (K1). Fix $(S,a)$ and $(T,b)$
satisfying (K1) and (K2) and $U$, $V$ as above. Let $U'=\{(S',a')\mid (\forall
r\preceq s) r\in S'$ if and only if $r\in S\}$ and $V'=\{(T',b')\mid (\forall
r\preceq t) r\in T'$ if and only if $r\in T\}$. These two sets are open and for
$(S',a')\in U'$ and $(T',b')\in V'$ we have $S'\cap T'=S\cap T$.
 For each $i$ in this intersection
  $\bP_{\{i\}}$ has finite rank and in a finite-dimensional space the norm topology
coincides with the strong operator topology, therefore (K2) is open on $\cX$
modulo (K1).

It remains to prove   (K3) is open. Assuming the pair $\{(S,a),(T,b)\}$
satisfies one of the alternatives of (K3) (without a loss of generality, the
first one) one only needs to fix a unit vector $\xi$ such that
$\|(\Psi(a)\Psi(\bP_T)-\Psi(\bP_S)\Psi(b))\xi\|>2^{-n}$; this defines an open
neighborhood consisting of pairs satisfying (K3).
\end{proof}

\begin{claim} \lbl{C.K-n-0} There are no uncountable $K^n_0$-homogeneous sets  for any $n$.
\end{claim}

\begin{proof}
Assume the contrary. Fix $n\in \bbN$ and an uncountable
$K^n_0$-ho\-mo\-ge\-ne\-ous~$\bfH$.
 For  $i\in M=\bigcup_{(S,a)\in \bfH} S$ fix $(S_i,a_i)\in
\bfH$ such that $i\in S_i$ and  let $c=\sum_{i\in M} a_i\bP_{\{i\}}$. Then
$c\in \calD_M[\vec E]_{\leq 1}$ and
$\|(c-a)\bP_{\{i\}}\|=\|(a_i-a)\bP_{\{i\}}\|<2^{-i}$ for all $(S,a)\in \bfH$.
For $(S,a)\in \bfH$ we have $M\supseteq S$ and the operator $\bP_Sc-a=c\bP_S-a$
is compact. Therefore, the operators $\Psi(c)\Psi(\bP_S)-\Psi(a)$ and
$\Psi(\bP_S)\Psi(c)-\Psi(a)$ are in $\cK(H)$. There is a finite-dimensional
projection $\bfR=\bfR(S,a)$ such that
$\|(I-\bfR)(\Psi(c)\Psi(\bP_S)-\Psi(a))\|<2^{-n-2}$ and
$\|(I-\bfR)(\Psi(\bP_S)\Psi(c)-\Psi(a))\|<2^{-n-2}$. Since $\Psi(\bP_S)$ is a
projection, we may choose $\bfR$ so that $\bfR\Psi(\bP_S)=\Psi(\bP_S)\bfR$.

Let $\delta=2^{-n-4}$.  By the separability of $\cK(H)$ there is a projection
$\bbfR$ and an uncountable $\bfH'\subseteq \bfH$ such
$\|\bbfR-\bfR(S,a)\|<\delta$ for all $(S,a)$ in $\bfH'$. By the
norm-separability of the range of $\bbfR$ we may find an uncountable
$\bfH''\subseteq \bfH'$ such that for all $(S,a)$ and $(T,b)$ in $\bfH''$ we
have $\|\bbfR(\Psi(\bP_S)-\Psi(\bP_T))\|<\delta$ and
$\|\bbfR(\Psi(a)-\Psi(b))\|<\delta$.

Write $a\approx_\e b$ for $\|a-b\|<\e$. Fix distinct $(S,a)$ and $(T,b)$ in
$\bfH''$. Recalling that $\|\Psi(d)\|=\|d\|$ for all $d$, we have
\begin{align*}
(I-\bbfR)\Psi(a)\Psi(\bP_T)&\approx_\delta &(I-\bfR(S,a))\Psi(a)\Psi(\bP_T)\\
&\approx_{2^{-n-2}}& (I-\bfR(S,a))\Psi(\bP_S)\Psi(c)\Psi(\bP_T)\\
&=&\Psi(\bP_S)(I-\bfR(S,a))\Psi(c)\Psi(\bP_T)\\
&\approx_{2\delta}&\Psi(\bP_S)(I-\bfR(T,b))\Psi(c)\Psi(\bP_T)\\
&\approx_{2^{-n-2}}&\Psi(\bP_S)(I-\bfR(T,b))\Psi(b)\\
&\approx_{2\delta}&\Psi(\bP_S)(I-\bfR(S,a))\Psi(b)\\
&=&(I-\bfR(S,a))\Psi(\bP_S)\Psi(b)\\
&\approx_\delta &(I-\bbfR)\Psi(\bP_S)\Psi(b),
\end{align*}
hence $\|(I-\bbfR)(\Psi(a)\Psi(\bP_T)-\Psi(\bP_S)\Psi(b))\|<6\delta+2^{-n-1}$.
Also
$$
\bbfR\Psi(a)\Psi(\bP_T)\approx_\delta\bbfR\Psi(b)\Psi(\bP_T)=\bbfR\Psi(\bP_T)\Psi(b)\approx_\delta
\bbfR\Psi(\bP_S)\Psi(b)
$$
 and
$\|\Psi(a)\Psi(\bP_T)-\Psi(\bP_S)\Psi(b)\|<8\delta+2^{-n-1}<2^{-n}$. Since an
analogous argument shows $\|\Psi(\bP_T)\Psi(a)-\Psi(b)\Psi(\bP_S)\|<2^{-n}$,
the pair $\{(S,a),(T,b)\}$ satisfies (K3). Since (K1) and (K2) are automatic we
have  $\{(S,a),(T,b)\}\in K^n_1$, a contradiction.
\end{proof}

With  $k$ as in the statement of Lemma~\ref{L.cJ-m} let $\bar n=k+3$.
By Claim~\ref{C.K-n-0} and TA, $\cX$ can be covered by the union of
$K^{\bar n}_1$-homogeneous sets $\cX_i$ for $i\in \bbN$. For each $i$
fix a countable $\bfD_i\subseteq \cX_i$ dense in the separable metric
topology from Claim~\ref{C0}. It will suffice to prove that every
$B\in \cA\setminus\{ B(S): (\exists a)(S,a)\in \bigcup_i \bfD_i\}$
belongs to $\cJ^{\bar n-3}_\sigma(\vec E)$.


Fix a dense set of projections  $Q_i$, for $i\in \bbN$, in the
projections of $\cK(H)$. We also assume that $Q_0=0$ and that for
every $i$ the set $\{Q_m: Q_m\geq Q_i\}$ is dense in $\{P: P\geq Q_i$
and $P$ is a projection in $\cK(H)\}$. For example, we may let $Q_m$
enumerate all finite rank projections belonging to some countable
elementary submodel of $H_{\mathfrak c^+}$.

For $m\in \bbN$ define a relation $\sim_m$ on $\cX$ by letting
\[
(S,a)\sim_m (T,d)
\]
 if and only if all of the following conditions
are satisfied.
\begin{enumerate}
\item [($\sim_m$1)] $S\cap m=T\cap m$,
\item [($\sim_m$2)] $\|(a-b)\bP_{\{i\}}\|<2^{-i-1}$ for all $i<m$,
\item [($\sim_m$3)] $\|Q_j(\Psi(\bP_S)-\Psi(\bP_T))Q_j\|\leq 1/m$ for all $j\leq m$, and
\item [($\sim_m$4)] $\|Q_j(\Psi(a)-\Psi(b))Q_j\|\leq 1/m$ for all $j\leq m$.
\end{enumerate}
We should emphasize that this is not an equivalence relation.

For $p$ and $m$ in $\bbN$ and  $(S,a)\in \cX_{p}$ let
\[
m^+(S,a,p)=\min\{j>m: (\exists (T,d)\in \bfD_{p}) ( (T,d)\sim_m (S,a)
\text{ and } T\cap B\subseteq j)\}.
\]
If $(S,a)\in \cX_{p}$ then $(T,d)$ as in the definition of
$m^+(S,a,p)$ exists and $T\cap B$ is finite. Therefore $m^+(S,a,p)$
is well-defined whenever $(S,a)\in \cX_{p}$.

Let us check that for every $m$ and every $p$ there is a finite set
$F_m\subseteq \bfD_{p}$ such that for every $(S,a)\in \cX_{p}$ there
is $(T,d)\in F_m$ satisfying $(S,a)\sim_m (T,d)$. Clearly there are
only finitely many possibilities for $S\cap m$. The projections
$\bP_{\{i\}}$ and $Q_j$ are finite-dimensional and therefore the unit
ball of the range of any of these projections is totally bounded.
Finally, note that in ($\sim_m$2) we have $(a-b)
\bP_{\{i\}}=\bP_{\{i\}} (a-b)\bP_{\{i\}}$. Therefore for $m\in \bbN$
we have that
\[
m^+=\max\{m^+(S,a,p): (S,a)\in \cX_{p}\text{ for some }p<m\}
\]
is well-defined. Let $m(0)=0$ and $m(j+1)>m(j)^+$ for all $j$. Let
\[
B_0=B\cap \textstyle\bigcup_{j=0}^\infty [m(2j),m(2j+1))
\]
and find a non-decreasing sequence $k(j)$, for $j\in \bbN$, such that
the following conditions are satisfied.
\begin{enumerate}
\item\label{YY1}
 $\delta(j)=\|Q_{k(j)}\Psi(\bP_{B_0})-\Psi(\bP_{B_0})Q_{k(j)}\|$
satisfies $\lim_{j\to \infty} \delta(j)=0$,
\item \label{YY2} $k(j)\leq m(2j+1)$, and
\item \label{YY4} $Q_{k(j)}$ strongly converge to the identity.
\pushcounter
\end{enumerate}
Let us describe the construction of the sequence $k(j)$, for $j\in
\bbN$.  Since we can
write $R$ as a strong limit of an increasing sequence of finite rank
projections there is an increasing sequence of finite rank
projections $R_i$, for $i\in \bbN$, that strongly converge to the
identity and such that 
\[
\lim_{j\to
\infty}\|R_j\Psi(\bP_{B_0})-\Psi(\bP_{B_0})R_j\|=0.
\]
 Let $k(0)=0$ and
using the density of~$Q_i$, for $i\in \bbN$, pick a nondecreasing
sequence $l(j)$ such that $\|Q_{l(j)}-R_j\|\to 0$  as $j\to\infty$
and $Q_{l(j)}$ converge to the identity in the strong operator
topology as $j\to \infty$. Letting $k(j)=\max\{l(i): l(i)\leq
m(2j+1)\}$ we have that \eqref{YY1}--\eqref{YY4} hold.

 For $a\in
\cU_{B_0}[\vec E]$ and $p\in \bbN$ let
\begin{align*}
\cY_{a,p}=\{c: (\forall j>p)(\exists (S,d)\in &\bfD_{p})\text{ so that }\\
(i)\ &S\cap B_0\subseteq m(2j+1), \\
(ii)\ &S\cap m(2j+1)=B_0\cap m(2j+1),\\
(iii)\ &\|(a-d)\bP_{\{i\}}\|<2^{-i}\text{ for }i\in S\cap B_0,\\
\text{ and for all $l\leq m(2j+1)$ we have }(iv)\ &\|Q_{l}(\Psi(\bP_{B_0})-\Psi(\bP_S))Q_{l}\|<2/j\\
\text{ and }(v)\ &\|Q_{l}(c-\Psi(d))Q_{l}\|<2/j\}.
\end{align*}
Since $\bfD_{p}$ is countable  the set
\[
\cY(\bar n, p)=\bigcup\{\{a\}\times \cY_{a,p} : a\in \cU_{B_0}[\vec
E]\}
\]
 is Borel for all $p$.

\begin{claim}\lbl{C7.5}
Assume  $a\in \cU_{B_0}[\vec E]$ is such that  $(B_0,a)\in \cX_{p}$.
Then
\begin{enumerate}
\popcounter
\item\label{C7.5.1} $\Psi(a)\in \cY_{a,p}$ and
\item \label{C7.5.2} $\|\Psi(\bP_{B_0})c-\Psi(a)\Psi(\bP_{B_0})\|<2^{-\bar n+1}$
for all $c\in \cY_{a,p}$. \pushcounter
\end{enumerate}
\end{claim}

\begin{proof} \eqref{C7.5.1} Fix $j$.
By the definition of $\sim_{m(2j+1)}$ and the choice of $m(2j+2)$ we
can choose $(S,d)\in\bfD_p$ such that (i)--(v) are satisfied with
$c=\Psi(a)$.

\eqref{C7.5.2} Assume the contrary, that $\|\Psi(\bP_{B_0})c
-\Psi(a)\Psi(\bP_{B_0})\|>2^{-\bar n+1}$. Fix~$j$ large enough to
have $2\delta(j)<2^{-\bar n}$ and
\begin{enumerate}
\popcounter
\item
$\|Q_{k(j)}(\Psi(\bP_{B_0})c-\Psi(a)\Psi(\bP_{B_0}))Q_{k(j)}\|>2^{-\bar
n+1}.$ \pushcounter
\end{enumerate}
Fix $i\geq j$.
 By the definition of $\cY_{a,p}$
we can pick $(S,d)=(S(i),d(i))\in \bfD_p$ such that
\begin{enumerate}
\popcounter
\item \label{XX1} $S\cap B_0\subseteq m(2i+1)$,
\item \label{XX2} $S\cap m(2i+1)=B_0\cap m(2i+1)$,
\item \label{XX3} $\|(a-d)\bP_{\{r\}}\|<2^{-r}$ for all $r\in S\cap B_0$,
\item\label{YY8}
$\|Q_{l}(\Psi(\bP_{B_0})-\Psi(\bP_S))Q_{l}\|<2/i$ for all $l\leq
m(2i+1)$, and
\item\label{YY10} $\|Q_{l}(c-\Psi(d))Q_{l}\|<2/i$ for all $l\leq m(2i+1)$.
\pushcounter
\end{enumerate}
Since the pair $\{(B_0,a),(S,d)\}$ belongs to $K^{\bar n}_1$ and the
corresponding instances of (K1) and (K2) hold, we must have
\begin{enumerate}
\popcounter
\item\label{YY15} $\|\Psi(\bP_{B_0})\Psi(d)-\Psi(a)\Psi(\bP_S)\|<2^{-\bar n}$.
\pushcounter
\end{enumerate} The proof is concluded by a computation.
 Writing  $x\approx^j_\e y$ for
\[
\|Q_{k(j)} (x-y)Q_{k(j)}\|\leq\e,
\]
by \eqref{YY1}, \eqref{YY10}, \eqref{YY1} and \eqref{YY15} respectively we have
\begin{multline*}
\Psi(\bP_{B_0})c\approx^j_{\delta(j)} \Psi(\bP_{B_0}) Q_{k(j)} c
\approx^j_{2/i} \Psi(\bP_{B_0})Q_{k(j)} \Psi(d)\\
\approx^j_{\delta(j)} \Psi(\bP_{B_0})\Psi(d)\approx^j_{2^{-\bar n}}
\Psi(a)\Psi(\bP_S)
\end{multline*}
and therefore
\begin{enumerate}
\popcounter
\item \label{YY16}$\|Q_{k(j)}(\Psi(\bP_{B_0})c-\Psi(a)\Psi(\bP_S)
)Q_{k(j)}\|\leq 2^{-\bar n}+\frac 2i+2\delta(j)$.
\end{enumerate}
Recall that $(S,d)=(S(i),d(i))$ depends on $i$ and note that
\eqref{YY8} implies that  $\Psi(\bP_{S(i)})$ converge to
$\Psi(\bP_{B_0})$ in the strong operator topology as $i\to \infty$.
Since the range of $Q_{k(j)}$ is finite-dimensional,
\[
\lim_{i\to \infty} \|Q_{k(j)}(\Psi(a)\Psi(\bP_{S(i)})-
\Psi(a)\Psi(\bP_{B_0}))Q_{k(j)}\|=0.
\]
Together with \eqref{YY16} this implies
\[
\|Q_{k(j)}(\Psi(\bP_{B_0})c-\Psi(a)\Psi(\bP_{B_0})
)Q_{k(j)}\|<2^{-\bar n+1},
\]
 a contradiction.
\end{proof}

By Theorem~\ref{T.JvN} there is a C-mea\-su\-ra\-ble uniformization
$\Theta^0_{p}\colon \cU_{B_0}[\vec E]\to \cB(H)$ of $\cY(\bar n, p)$.
By Claim~\ref{C7.5} the graphs of functions
\[
\Upsilon^0_{p}(a)=\Psi(\bP_{B_0})\Theta^0_{p}(a)
\]
for $p\in \bbN$ cover a graph of a $2^{-\bar n+1}$-approximation
to~$\Phi$. By Lemma~\ref{L.guess} \eqref{L.guess.1} there are
Borel-measurable functions wit\-nes\-sing $B_0\in \cJ^{\bar
n-2}_\sigma(\vec E)$. An analogous argument gives $(\Upsilon_i^1)_i$
witnessing $B_1=B\setminus B_0\in \cJ^{\bar n-2}_\sigma(\vec E)$.
Since $a\in \cU_B[\vec E]$ implies that both
$a\bP_{B_0}=\bP_{B_0}a\bP_{B_0}\in\dom(\Upsilon^0_i)$ and $a\bP_{
B_1}=\bP_{B_1}a\bP_{B_1}\in \dom(\Upsilon^1_j),$ functions
$\Upsilon_{ij}(a)=\Upsilon^0_i(a \bP_{B_0})+\Upsilon^1_j(a \bP_{
B_1})$ witness $B\in \cJ^{\bar n-3}_\sigma(\vec E)$.
\end{proof}

\subsection{Uniformizations}\lbl{S.sigma}
An automorphism $\Phi$ of $\cC(H)$ and its representation $\Psi$ are fixed. The
unitary group $\cU_A[\vec E]$ of $\calD_A[\vec E]$  is compact metric with
respect to its strong operator topology.  Let $\nu_{\vec E}$ denote the
normalized Haar measure on this group.

\begin{lemma}\lbl{L.F.B1}
Assume $\bK$ is a positive Haar-measurable subset of $\cU[\vec E]$ such that
$\Phi$ has a measurable $\e$-approximation $\Xi$ on $\bK$. Then $\Phi$ has a
Borel-measurable $2\e$-approximation on $\cU[\vec E]$.
\end{lemma}

\begin{proof} By Luzin's theorem (\cite[Theorem~17.12]{Ke:Classical}),
by possibly shrinking $\bK$ we may assume it is compact and the restriction of
$\Xi$ to $\bK$ is continuous. Let us first see that we may assume
$\nu(\bK)>1/2$. Let $\bU\subseteq \cU [\vec E]$ be a basic open set such that
$\nu(\bK\cap \bU)>\nu(\bU)/2$. Let $n$ be large enough so that there is an open
$\bU_0\subseteq \prod_{i<n} \cU(E_i)$ satisfying $\bU=\bU_0\times \prod_{i\geq
n} \cU(E_i)$. Fix a finite $F\subseteq \{a\in \cU[\vec E]\mid a(i)=I_i$ for all
$i\geq n\}$ such that $F\bU_0=\cU[\vec E]$. Then $\bK'=F\bK$ has measure $>1/2$
and $\Xi'$ with domain $\bK'$ defined by $\Xi'(b)=\Xi(ab)$, where $a$ is the
first element of $F$ such that $ab\in \bK$, is a continuous $\e$-approximation
of $\Phi$ on $\bK'$.

 Let $\cX=\{(a,b)\in \cU [\vec E]\times \bK'\mid ab^*\in
\bK'\}$. This set is closed.  Since~$\nu_{\vec E}$ is invariant and unimodular,
for each $a$ there is $b$ such that $(a,b)\in \cX$. By Theorem~\ref{T.U.Large.Measure}
there is a Borel-measurable $f\colon\cU [\vec E]\to \bK$ such that $(a,f(a))\in
\cX$ for all $a$. The map $\Xi_1(a)=\Xi(af(a)^*)\Xi(f(a))$ is clearly a
$2\e$-approximation to $\Phi$ and it is Borel-measurable.
\end{proof}

\begin{prop} \lbl{L.sigma}
If $M_i$, $i\in \bbN$ are pairwise disjoint infinite subsets of
$\bbN$ and $M=\bigcup_i M_i$ is in
 $\cJ^{n}_\sigma(\vec E)$
then there is $i$ such that $M_i\in \cJ^{n-2}(\vec E)$.
\end{prop}

\begin{proof} Assume not. Write $P_i=\bP^{\vec E}_{M_i}$ and $P=\bP^{\vec E}_M$.
Fix Borel-measurable functions $\Xi_i$, $i\in \bbN$,  whose graphs cover a
$2^{-n}$-ap\-pro\-xi\-ma\-ti\-on to $\Psi$ on $\cU_M[\vec E]$.  Let
$Q_i=\bigvee_{j=i}^\infty P_j$, hence $Q_0=P$. By making unessential changes to
$\Psi$, we may assume $\Psi(P_i)$,  $i\in \bbN$,  are pairwise orthogonal
projections such that $\Psi(Q_i)=\bigvee_{j\geq i} \Psi(P_j)$ for all $i$. Let
$\cV_i=\prod_{j=i}^\infty \cU_{M_i}[\vec E]$, a compact group with Haar measure
$\mu_i$. We shall find $a_i\in \cU_{M_i}[\vec E]$ and a $\mu_i$-positive compact
$\cY_i\subseteq \cV_{i+1}$ such that for all $i$ and all $b\in \cY_i$ we have
\begin{equation}
\label{E.inductive.1} \|(\Xi_i(\textstyle(\sum_{j\leq i}
a_j)+b)-\Psi(a_i))\Psi(P_i)\|_{\cK}>2^{-n}.
\end{equation}
We shall also assure that for $j<i$ we have
\begin{equation}
\lbl{E.inductive.2} \cY_i\subseteq \{b\in \cV_{i+1}\mid
b+\textstyle\sum_{k=j+1}^i a_k\in \cY_j\}.
\end{equation}
The condition \eqref{E.inductive.2} will assure that  
$\hat\cY_i\subseteq \cV_0$ defined 
by $\hat\cY_i=\{\sum_{j=0}^i a_j+b\colon b\in \cY_i\}$, for $i\in \bbN$, 
form a decreasing sequence of compact sets. 
Assume $a_0,a_1,\dots a_{i-1}$ and $\cY_{i-1}$ have been chosen to satisfy
\eqref{E.inductive.1} and \eqref{E.inductive.2}. Using Fubini's theorem,  
Lebesgue density theorem, and the 
inner regularity of the Haar measure, 
find compact  positive sets $V_i\subseteq \cU_{M_i}[\vec
E]$ and $W_i\subseteq \cV_{i+1}$ such that for every $x\in V_i$ we have
$\mu_{i+1}\{y\in W_i\mid (x,y)\in \cY_i\}>\mu_{i+1}(W_i)/2$. Let
\begin{multline*}
\cX=\{(a,b,c)\in V_i\times W_i\times \cU(H)\mid
\|\Xi_i((\textstyle\sum_{j<i}a_j)+a+b)\Psi(P_i)-c\|_{\cK}\leq 2^{-n}\}.
\end{multline*}
This is  a Borel set, and so is
$$
\cX_1=\{(a,c)\mid \mu_i\{b\mid (a,b,c)\in
\cX\}>\mu_i(W_i)/2\}.
$$
Let $\cZ$ be the set of all $a\in V_i$ such that $\{b\mid (a,b)\in \cX_1\}\neq
\emptyset$. This is a projection of $\cX_1$. If $\cZ\neq V_i$, pick $a_i\in
V_i\setminus \cZ$. Since $(a_i,\Psi(a_i))\notin \cX_1$, with
$$
\cY_{i+1}'=\{b\in W_i\mid
\|\Xi_i((\textstyle\sum_{j<i}a_j)+a_i+b)\Psi(P_i)-\Psi(a_i)\|_{\cK}> 2^{-n}\}
$$
we have $\mu_{i+1}(\cY_{i+1}')\geq \mu_{i+1}(W_i)/2$. In this case
$\cY_{i+1}''=\cY_{i+1}'\cap\{b\in W_i\mid (a_i,b)\in \cY_i\}$ is $\mu_i$-positive
and satisfies \eqref{E.inductive.1} and \eqref{E.inductive.2}. By the inner regularity of the Haar measure
find a compact positive $\cY_{i+1}\subseteq \cY_{i+1}''$ and  
proceed with the construction.

We may therefore without a loss of generality assume $\cZ=V_i$.
 By Theorem~\ref{T.JvN} there is a C-measurable $\bar f\colon \cU_{M_i}[\vec E]\to \cB(H)$
such that $(a,\bar f(a))\in \cX_1$ for all $a\in \cZ$. 
Then $f$ defined by $f(a)=\Psi(P_i)\bar f \Psi(P_i)$ is also Borel. 
Since $\cZ=V_i$ has positive
measure, if $f$ is a $2^{-n+1}$-approximation of $\Phi$ on $\cZ$, then
Lemma~\ref{L.F.B1} gives a Borel $2^{-n+2}$-approximation of $\Phi$ on
$\cU_{M_i}[\vec E]$, showing that $M_i\in \cJ^{n-2}(\vec E)$ and contradicting
our assumption. Therefore we can fix $a_i\in \cZ$ such that
$\|(f(a_i)-\Psi(a_i))\Psi(P_i)\|_{\cK}>2^{-n+1}$. Then $\cY_{i+1}'=\{b\mid
(a_i,b,f(a_i))\in \cX\}$ has a positive measure and for each
$b\in \cY_{i+1}'$ clause \eqref{E.inductive.1} holds because
\begin{multline*}
\|(\Xi_i(\textstyle\sum_{j\leq i}
a_j+b)-\Psi(a_i))\Psi(P_i)\|_{\cK} \\
\geq \|(f(a_i)-\Psi(a_i))\Psi(P_i)\|_{\cK}-
\|(\Xi_i(\textstyle\sum_{j\leq i}
a_j+b)-f(a_i))\Psi(P_i)\|_{\cK}>2^{-n}.
\end{multline*}
Let $\cY_{i+1}\subseteq \cY_{i+1}'$ be a compact positive set. 
 This describes
the construction. Let $a=\sum_{i=0}^\infty a_i$. Since
$a_i=P_iaP_i$ for each $i$ and $P_i$ are pairwise
orthogonal, $\|a\|\leq \sup_i \|a_i\|=1$.
 For some $i$ we have
$\|\Xi_i(a)-\Psi(a)\|_{\cK}\leq 2^{-n}$, hence
$\|(\Xi_i(a)-\Psi(a_i))\Psi(P_i)\|_{\cK}\leq 2^{-n}$. However, $\sum_{j=i+1}^\infty a_i$ is in $\cY_i$ 
by~\eqref{E.inductive.2}  and the compactness of  $\cY_i$ in the product topology. 
This contradicts \eqref{E.inductive.1}.
\end{proof}

\begin{proof}[Proof of Proposition~\ref{P.2}] Enumerate $\bbN$ as $n_s$ ($s\in
2^{<\bbN}$) and write $M_x=\{n_{x\rs j}\mid j\in \bbN\}$. By
Lemma~\ref{L.cJ-m}, for every $m$ the set $\{x\mid M_x\notin
\cJ^{n}_\sigma(\vec E)\}$ is at most countable. We may therefore fix
$x_0$ such that $M_0=M_{x_0}$ belongs to $\cJ^{n}_\sigma(\vec E)$ for
each $n$. Partition $M_0$ into infinitely many infinite pieces. By
Lemma~\ref{L.sigma} at least one of these pieces, call it $M_1$,
belongs to $\cJ^1(\vec E)$. By successively applying this argument we
find a decreasing sequence $M_j$ of infinite subsets of $M_0$ such
that $M_j\in \cJ^j(\vec E)$ for each $j$. Fix an infinite $M$ such
that $M\setminus M_j$ is finite for all $j$. Then $M\in \bigcap_j
\cJ^j(\vec E)$ and on $\calD_M[\vec E]$ there is a Borel-measurable
$2^{-j}$-approximation to $\Phi$ for each $j$. By Lemma~\ref{L.n}
there is a C-measurable representation of $\Phi$ on $\calD_M[\vec
E]$. By Theorem~\ref{T.2}, $\Phi$ is inner on $\calD_M[\vec E]$ and
by Lemma~\ref{L.trivial.2}, $\Phi$ is inner on $\calD[\vec E]$.
\end{proof}

\begin{proof}[Proof of Theorem~\ref{T.0}] Fix an automorphism $\Phi$ of
$\cC(H)$ and an orthonormal basis $(e_n)$ for $H$. For every partition $\vec E$
of $\bbN$ into finite intervals such that $\# E_n$ is nondecreasing
Proposition~\ref{P.2} implies there is a partial isomorphism $u=u(\vec E)$
between cofinite-dimensional subspaces of $H$ such that $\Psi_u$ is a
representation of $\Phi$ on $\cC[\vec E]$. Therefore $\{(\vec E,u(\vec E))\}$
is a coherent family of unitaries and Theorem~\ref{P.1} implies $\Phi$ is
inner.
\end{proof}

\section{Concluding remarks}
 Let $S$ denote the unilateral shift
operator. The following problem of Brown--Douglas--Fillmore is
well-known.

\begin{problem} \lbl{Q.1}
Is it consistent with the usual axioms of mathematics that some automorphism of
the Calkin algebra sends $\pi(S)$ to its adjoint?
\end{problem}

Ilan Hirshberg pointed out that there are essentially normal
operators $a$ and $b$ with the same essential spectrum such that
$\Phi(\pi(a))\neq  \pi(b)$ for all inner automorphisms $\Phi$ of the
Calkin algebra. This is because for a fixed $\Phi$ either
$\findex(\Phi(a))=\findex(a)$ for all Fredholm operators $a$ or
$\findex(\Phi(a))=-\findex(a)$ for all Fredholm operators $a$.
Together with the Brown--Douglas--Fillmore characterization of
unitary equivalence modulo compact perturbation of essentially normal
operators, this implies that a positive answer to Problem~\ref{Q.1}
is equivalent to the consistency of the existence of normal operators
$a$ and $b$ in $\cC(H)$ and an automorphism $\Phi$ of $\cC(H)$ such
that $\Phi(a)=b$ but for every inner automorphism $\Psi$ of $\cC(H)$
we have $\Psi(a)\neq b$.
An argument using \cite{AlCoMac} shows that if an automorphism $\Phi$
sends the standard atomic masa to itself then $\Phi$ cannot send $\dot S$ to $\dot S^*$
(see \cite[Proposition~7.7]{FaWo:Set}).

Recall that for a C*-algebra $A$ its \emph{multiplier algebra}, the
quantized analogue of the \v Cech--Stone compactification,  is
denoted by $M(A)$  (see \cite[1.7.3]{Black:Operator}). For example,
$M(\cK(H))=\cB(H)$, $M(C_0(X))=C(\beta X)$ for a locally compact
Hausdorff space $X$, and $M(A)=A$ for every unital C*-algebra $A$.
George Elliott suggested investigating when all automorphisms of
$M(A)/A$ are trivial and Ping Wong Ng suggested investigating when
isomorphism of the corona algebras $M(A)/A$ and $M(B)/B$ implies
isomorphism of $A$ and $B$. The following is the set-theoretic core
of both of these problems and it is very close to \cite{Fa:Rigidity}
and \cite{Fa:AQ} in spirit.

\begin{problem} Assume  $A$ and $B$ are separable non-unital C*-algebras.
When does every isomorphism between the corona algebras $M(A)/A$ and
$M(B)/B$ lift to a
*-homomorphism $\Phi$ of  $M(A)$ into $M(B)$, so that the
diagram \[ \diagram
M(A) \xto[r]^{\Phi}\xto[d]^{\pi} &   M(B)\xto[d]^{\pi}\\
M(A)/A \xto[r]_{\Psi} &M(B)/B
\enddiagram
\]
commutes?
\end{problem}

TA implies the positive answer  when $A=B=\cK(H)$
(Theorem~\ref{T.0}) and TA+MA implies the positive answer when both
$A$ and $B$ are of the form~$C_0(X)$ for a countable locally compact
space $X$ (\cite[Chapter 4]{Fa:AQ}). One could also ask analogous
questions for
*-homomorphisms instead of isomorphisms or, as suggested by Ping Wong Ng,  for
$\ell^\infty(A)/c_0(A)$ instead of the corona algebra. 
A number of analogous lifting results for 
quotient Boolean algebras $\cP(\bbN)/\cI$ was
proved in \cite{Fa:AQ} 
(see also \cite{Fa:Luzin, Fa:Rigidity}).

It was recently proved by the author,  Schimmerling and McKenney that 
the Proper Forcing Axiom, PFA, implies all automorphisms of the Calkin algebra
$\cB(H)/\cK(H)$ are inner, even for nonseparable  Hilbert spaces. 
 An analogous result for  automorphisms of the Boolean algebra 
$\cP(\kappa)/\Fin$, where $\kappa$ is arbitrary, was proved in
\cite{Ve:OCA}.

\subsection*{Acknowledgments} I would like to thank George Elliott, N. Christopher Phillips,
Efren Ruiz,
Juris Step\-r\=ans, Nik Weaver and  Eric Wofsey for stimulating conversations.
 I would also like
to thank J\"org Brendle, Stefan Geschke, Saeed Ghasemi, 
Ilan Hirshberg, Paul McKenney, N. Christopher
Phillips, Sh{\^o}ichir{\^o} Sakai, Ernest Schimmerling,  and Nik Weaver for valuable
comments on earlier versions of this paper. I am
particularly indebted to J\"org Brendle and Paul McKenney. 
J\"org  presented a large part of
the proof in a series of seminars at the Kobe University and suggeted
several improvements in the proof of Lemma~\ref{L.cJ-m}. 
Paul presented the entire proof over the course of sixteen 90$^+$ minute lectures
to an attentive audience at the Mittag--Leffler Institute in September 2009. 
 Finally, I would like to thank Stevo Todorcevic for persistently
insisting that the techniques developed in~\cite{Fa:AQ} will find
other applications.

\providecommand{\bysame}{\leavevmode\hbox to3em{\hrulefill}\thinspace}
\providecommand{\MR}{\relax\ifhmode\unskip\space\fi MR }
\providecommand{\MRhref}[2]{%
  \href{http://www.ams.org/mathscinet-getitem?mr=#1}{#2}
}
\providecommand{\href}[2]{#2}

\end{document}